\documentclass[reqno,a4paper,12pt]{amsart}
\usepackage[utf8]{inputenc}

\calclayout

\usepackage[utf8]{inputenc}
\usepackage{verbatim}
\usepackage[english]{babel}
\usepackage{amsmath}
\usepackage{amsfonts}
\usepackage{amsthm}
\usepackage{dirtytalk}
\usepackage{float}
\usepackage{color}
\usepackage{tikz}
\usetikzlibrary{decorations.pathreplacing,decorations.markings}
\usepackage{graphicx}
\usepackage{subcaption}
\usepackage{enumerate}
\usepackage{seqsplit}

\usepackage[pdftex,colorlinks]{hyperref}
\hypersetup{
	colorlinks=true,
	linkcolor=blue,
	citecolor=red
}
\tikzset{
	on each segment/.style={
		decorate,
		decoration={
			show path construction,
			moveto code={},
			lineto code={
				\path [#1]
				(\tikzinputsegmentfirst) -- (\tikzinputsegmentlast);
			},
			curveto code={
				\path [#1] (\tikzinputsegmentfirst)
				.. controls
				(\tikzinputsegmentsupporta) and (\tikzinputsegmentsupportb)
				..
				(\tikzinputsegmentlast);
			},
			closepath code={
				\path [#1]
				(\tikzinputsegmentfirst) -- (\tikzinputsegmentlast);
			},
		},
	},
	mid arrow/.style={postaction={decorate,decoration={
				markings,
				mark=at position .6 with {\arrow[#1]{stealth}}
	}}},
	rmid arrow/.style={postaction={decorate,decoration={
				markings,
				mark=at position .4 with {\arrowreversed[#1]{stealth}}
	}}},
}

\makeatletter
\def\grd@save@target#1{%
  \def\grd@target{#1}}
\def\grd@save@start#1{%
  \def\grd@start{#1}}
\tikzset{
  grid with coordinates/.style={
    to path={%
      \pgfextra{%
        \edef\grd@@target{(\tikztotarget)}%
        \tikz@scan@one@point\grd@save@target\grd@@target\relax
        \edef\grd@@start{(\tikztostart)}%
        \tikz@scan@one@point\grd@save@start\grd@@start\relax
        \draw[minor help lines] (\tikztostart) grid (\tikztotarget);
        \draw[major help lines] (\tikztostart) grid (\tikztotarget);
        \grd@start
        \pgfmathsetmacro{\grd@xa}{\the\pgf@x/1cm}
        \pgfmathsetmacro{\grd@ya}{\the\pgf@y/1cm}
        \grd@target
        \pgfmathsetmacro{\grd@xb}{\the\pgf@x/1cm}
        \pgfmathsetmacro{\grd@yb}{\the\pgf@y/1cm}
        \pgfmathsetmacro{\grd@xc}{\grd@xa + \pgfkeysvalueof{/tikz/grid with coordinates/major step}}
        \pgfmathsetmacro{\grd@yc}{\grd@ya + \pgfkeysvalueof{/tikz/grid with coordinates/major step}}
        \foreach \x in {\grd@xa,\grd@xc,...,\grd@xb}
        \node[anchor=north] at (\x,\grd@ya) {\pgfmathprintnumber{\x}};
        \foreach \y in {\grd@ya,\grd@yc,...,\grd@yb}
        \node[anchor=east] at (\grd@xa,\y) {\pgfmathprintnumber{\y}};
      }
    }
  },
  minor help lines/.style={
    help lines,
    step=\pgfkeysvalueof{/tikz/grid with coordinates/minor step}
  },
  major help lines/.style={
    help lines,
    line width=\pgfkeysvalueof{/tikz/grid with coordinates/major line width},
    step=\pgfkeysvalueof{/tikz/grid with coordinates/major step}
  },
  grid with coordinates/.cd,
  minor step/.initial=.2,
  major step/.initial=1,
  major line width/.initial=0.25mm,
}

\makeatletter
\def\l@subsection{\@tocline{2}{0pt}{2.5pc}{5pc}{}}

\DeclareMathOperator{\diag}{diag}
\DeclareMathOperator{\re}{Re}
\DeclareMathOperator{\im}{Im}
\DeclareMathOperator{\supp}{supp}

\newcommand{\res}{\mathop{\rm Res}}

\newcommand{\N}{\mathbb{N}}
\newcommand{\C}{\mathbb{C}}
\newcommand{\R}{\mathbb{R}}
\newcommand{\Z}{\mathbb{Z}}

\newcommand{\boh}{\mathit{o}}
\newcommand{\Boh}{\mathcal{O}}

\usepackage{todonotes}

\newtheorem{theorem}{Theorem}[section]
\newtheorem{prop}[theorem]{Proposition}
\newtheorem{lemma}[theorem]{Lemma}
\newtheorem{corollary}[theorem]{Corollary}

\theoremstyle{definition}

\theoremstyle{remark}
\newtheorem{remark}[theorem]{Remark} 
\numberwithin{equation}{section} 
 
\newcommand\restr[2]{{
		\left.\kern-\nulldelimiterspace 
		#1 
		\vphantom{\big|} 
		\right|_{#2} 
}}

\begin{document}

\title{Supercritical Regime for the Kissing Polynomials}

\author[A.~Celsus]{Andrew F. Celsus}

\address[AC]{University of Cambridge, Cambridge, UK.}

\email{a.f.celsus@maths.cam.ac.uk }

\author[G.~Silva]{Guilherme L.~F.~Silva}

\address[GS]{{\it Corresponding Author}. Instituto de Ciências Matemáticas e de Computação, Universidade de São Paulo (ICMC - USP), São Carlos - SP, Brazil.} 

\email{silvag@usp.br}

\date{}


\begin{abstract}
We study a family of polynomials which are orthogonal with respect to the varying, highly oscillatory complex weight function $e^{ni\lambda z}$ on $[-1,1]$,  where $\lambda$ is a positive parameter. This family of polynomials has appeared in the literature recently in connection with complex quadrature rules, and their asymptotics have been previously studied when $\lambda$ is smaller than a certain critical value, $\lambda_c$. Our main goal is to compute their asymptotics when $\lambda>\lambda_c$. 

We first provide a geometric description, based on the theory of quadratic differentials, of the curves in the complex plane which will eventually support the asymptotic zero distribution of these polynomials. Next, using the powerful Riemann-Hilbert formulation of the orthogonal polynomials due to Fokas, Its, and Kitaev, along with its method of asymptotic solution via Deift-Zhou nonlinear steepest descent, we provide uniform asymptotics of the polynomials throughout the complex plane. 

Although much of this asymptotic analysis follows along the lines of previous works in the literature, the main obstacle appears in the construction of the so-called global parametrix. This construction is carried out in an explicit way with the help of certain integrals of elliptic type. In stark contrast to the situation one typically encounters in the presence of real orthogonality, an interesting byproduct of this construction is that there is a discrete set of values of $\lambda$ for which one cannot solve the model Riemann-Hilbert problem, and as such the corresponding polynomials fail to exist.
\end{abstract}


\vspace*{-1.6cm}

\maketitle

\tableofcontents

\section{Introduction}
	Polynomial sequences satisfying non-Hermitian or complex orthogonality conditions originally appeared in the literature due to their use in approximation theory \cite{aptekarev1992asymptotics,gonchar1989equilibrium,nikishin1991rational}, but have recently been used in their connection to theoretical physics and random matrix theory \cite{alvarez2013determination,alvarez2014partition,alvarez2015fine,bertola2011boutroux} and Painlev\'e equations \cite{balogh2016hankel,bertola_tovbis_2015,bertola2014zeros}, to mention only a few.

	In the context of the present work, the motivation for studying these polynomials comes from the problem of constructing quadrature rules for oscillatory integrals of the form
	\begin{equation}\label{eq: oscillatory integral}
	I_\omega\left[f\right]:= \int_{-1}^{1} f(z) e^{i\omega z}\, dz. 
	\end{equation}
	It was shown by Asheim, Deaño, Huybrechs and Wang \cite{asheim2012gaussian} that by using the zeros of the monic polynomial $p_n^\omega$ of degree $n$ satisfying the orthogonality conditions,
	\begin{equation}\label{eq: regular kissing polys}
	\int_{-1}^1 p_n^\omega(z) z^k e^{i\omega z}\, dz = 0, \qquad k = 0, 1, \dots, n-1,
	\end{equation}
	one could construct a quadrature rule for \eqref{eq: oscillatory integral} that not only attained high asymptotic order as $\omega \to \infty$, but also reduced elegantly to traditional Gauss-Legendre quadrature as $\omega \to 0$. We mean high asymptotic order in the following sense. Letting $\{x_j\}_{j=1}^{2n}$ be the zeros of $p_{2n}^\omega$ and $\{w_j\}_{j=1}^{2n}$ be the associated Gaussian Quadrature weights, \cite[Theorem~4.1]{asheim2012gaussian} tells us that for analytic $f$ we have the following asymptotic decay,
	\begin{equation}
		\sum_{j=1}^{2n} w_j f\left(x_j\right) - \int_{-1}^1 f(x)e^{i\omega x}\, dx = \mathcal{O}\left(\frac{1}{\omega^{2n+1}}\right), \qquad \omega\to\infty.
	\end{equation}
	This is indeed high order, especially when compared with typical asymptotic methods (such as Filon methods) which attain $\mathcal{O}\left(\omega^{-n-1}\right)$ as $\omega\to\infty$ using the same amount of information. More details on this form of complex Gaussian quadrature, along with other numerical methods for the treatment of highly oscillatory integrals, can be found in the recent monograph of Dea\~no, Huybrechs, and Iserles \cite{quad_book}. As a result of this work, it became an important question in numerical analysis to understand the zero behavior of the sequence of polynomials $(p_n^\omega)$, in particular under various possible asymptotic regimes in the parameter $\omega$ and also in the degree $n$.
	
	Although $(p_n^\omega)$ was, as mentioned, first studied in \cite{asheim2012gaussian}, it was Deaño, Huybrechs, and Iserles \cite{deano2015kissing} who dubbed them the \say{Kissing Polynomials} due to the patterns formed by their zero trajectories, see Figure~\ref{fig: kissing}. In the same work, they showed that for fixed degree and as $\omega \to \infty$, the zeros of the kissing polynomials tend to $\pm 1$. 
	\begin{figure}[t]
		\centering
		\begin{subfigure}[b]{0.32\linewidth}
			\includegraphics[width=\linewidth]{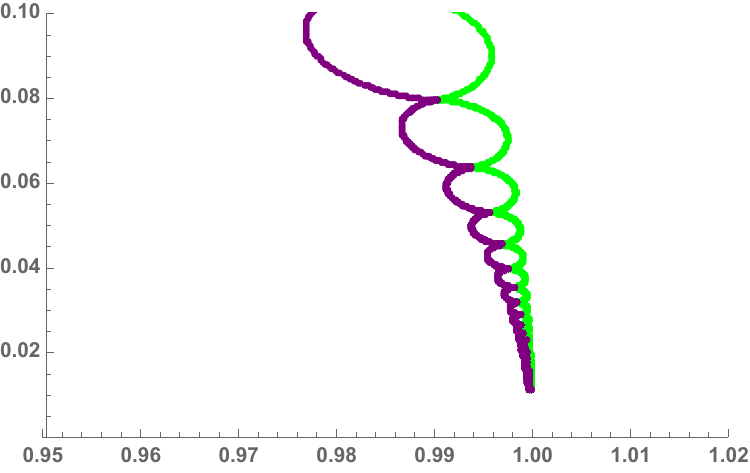}
			\caption*{$n = 2,3$}
			\label{fig: kissing 23}
		\end{subfigure}
		\begin{subfigure}[b]{0.32\linewidth}
			\includegraphics[width=\linewidth]{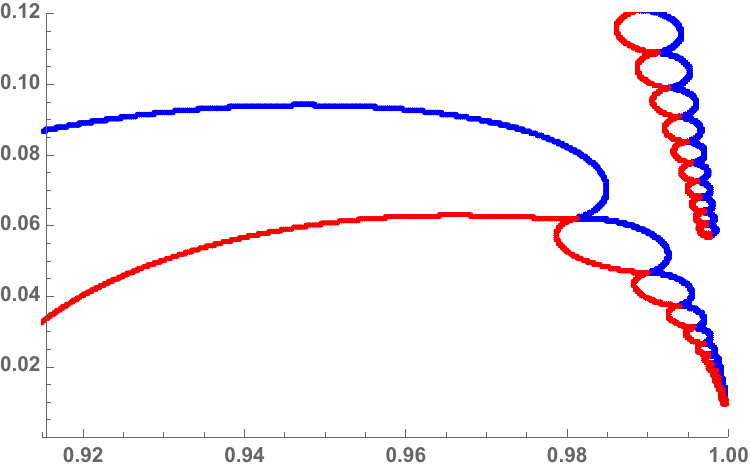}
			\caption*{$n = 4,5$}
			\label{fig: kissing 45}
		\end{subfigure}
		\caption{Zeros trajectories of the degree 2 (green), 3 (purple), 4 (blue), 5 (red) Kissing Polynomials as $\omega$ varies, zoomed in to the end point $+1$. The trajectories for polynomials of consecutive degree touch or ``kiss'', hence the name, each other.}
		\label{fig: kissing}
	\end{figure}
	The asymptotic analysis of $p^\omega_n$ for fixed $\omega$ and as $n \to \infty$ can be found in the appendix of \cite{deano2014large} or by using the techniques of \cite{kuijlaars2004riemann}, where one can show that the zeros of the kissing polynomials accumulate on the interval $[-1,1]$. 
	
	One can also let $\omega$ depend on $n$. To get to a nontrivial new limit, one sets $\omega = \omega(n) = \lambda n$ in \eqref{eq: regular kissing polys} which leads to the formulation of the varying weight kissing polynomials, which we formally reintroduce by
	\begin{equation}\label{eq: varying weight kissing polynomials}
	\int_{-1}^{1} p_n^\lambda(z) z^k e^{-nV(z)}\, dz = 0, \qquad k = 0, 1, \dots, n-1,\quad V(z):=-i \lambda z.
	\end{equation}
	Thus, studying the behavior of the kissing polynomials as $\omega$ and $n$ go to infinity at the rate $\lambda$ is equivalent to studying the large $n$ behavior of the varying weight kissing polynomials defined by \eqref{eq: varying weight kissing polynomials}. We stress that $p_n^\omega$ and $p_n^\lambda$ are not the same polynomials although they are easily related through a rescaling.
	
	In \cite{deano2014large}, Deaño studied the large degree asymptotics for $(p_n^\lambda)$, showing that for $\lambda<\lambda_c$ the zeros of $p_n^\lambda$ accumulate on a single analytic arc connecting $-1$ and $1$. As shown in \cite{deano2014large}, the critical value $\lambda_c$ (numerically, $\lambda_c \approx 1.32549$) is the unique positive solution to the equation
	\begin{equation}\label{eq: lambda crit defn}
	2 \log\left(\frac{2+\sqrt{\lambda^2+4}}{\lambda}\right) - \sqrt{\lambda^2+4}=0. 
	\end{equation}
	Deaño also noted that for $\lambda > \lambda_c$, the zeros of $p_n^\lambda$ seem to accumulate on two disjoint arcs in the complex plane, as illustrated in Figure~\ref{fig: intro pic}.
	
	\begin{figure}[t]
		\centering
		\begin{subfigure}[b]{0.32\linewidth}
			\includegraphics[width=\linewidth]{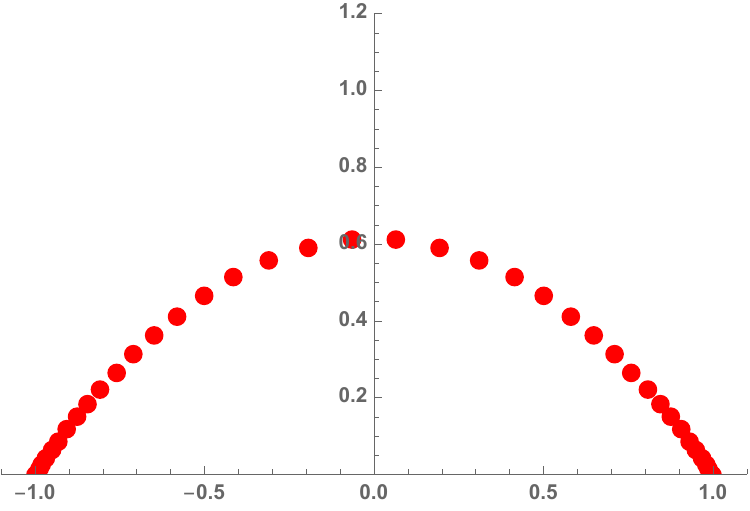}
			\caption*{$\lambda = 1$}
			\label{fig: n40 lam 1}
		\end{subfigure}
		\begin{subfigure}[b]{0.32\linewidth}
			\includegraphics[width=\linewidth]{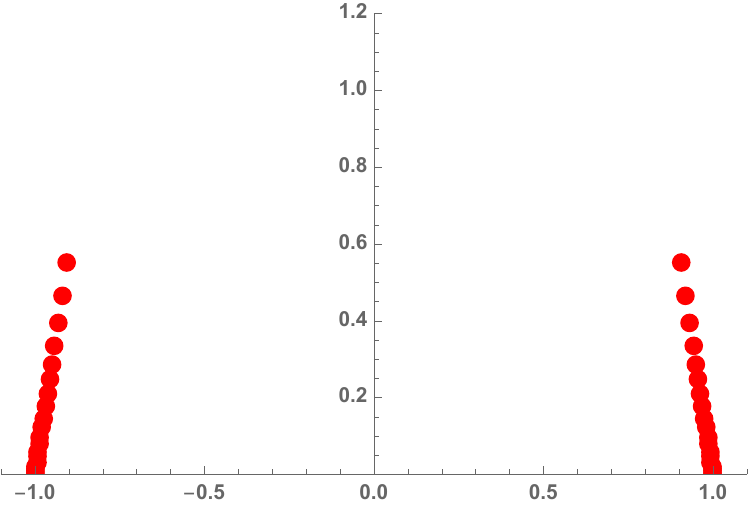}
			\caption*{$\lambda = 3$}
			\label{fig: n40 lam 3}
		\end{subfigure}
		\begin{subfigure}[b]{0.32\linewidth}
			\includegraphics[width=\linewidth]{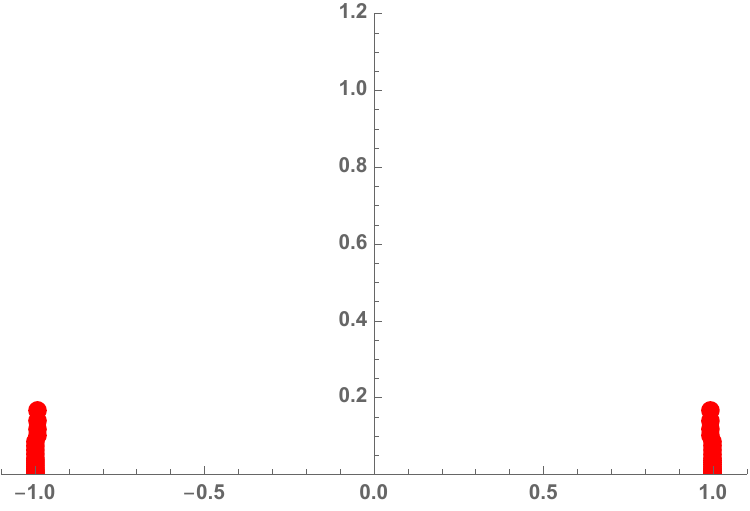}
			\caption*{$\lambda = 10$}
			\label{fig: n40 lam 10}
		\end{subfigure}
		\caption{Zeros of $p_{40}^\lambda(z)$ for $\lambda=1,3,10$.}
		\label{fig: intro pic}
	\end{figure}
	
	The goal of this paper is to prove that the situation depicted in Figure~\ref{fig: intro pic} is correct. Therefore, we will show that for $\lambda > \lambda_c$, the zeros of the varying weight kissing polynomials do indeed accumulate on two disjoint arcs, one emanating from $-1$ and the other emanating from $+1$. These arcs turn out to be analytic, and we will describe them precisely. We will also provide strong asymptotic formulas for $p_n^\lambda$ in the complex plane. By studying the supercritical regime, we are thus completing the picture regarding non-trivial asymptotics of the kissing polynomials. 

\section{Statement of Main Results}

	In this section, we provide the necessary background on complex orthogonality, setting up the notation used throughout this text, and discuss our main findings. 
	
	It should first be noted that due to the complex-valued nature of the weight function in \eqref{eq: varying weight kissing polynomials}, questions such as existence of the polynomials $p_n^\omega$ and bounds for the location of their zeros, which are taken for granted when dealing with real, positive weight functions, are no longer known a priori. Indeed, these questions were considered by Deaño, Huybrechs and Iserles in their study of the kissing polynomials \cite{deano2015kissing}. 
	
	
	As everything in the integrand of \eqref{eq: varying weight kissing polynomials} is analytic, we have complete freedom when choosing the path of integration connecting $-1$ and $+1$ in the orthogonality conditions \eqref{eq: varying weight kissing polynomials}. On the other hand, accounting for the asymptotic behavior of $p_n^\lambda$ as $n\to\infty$, and in particular its zeros, it is expected that there exists a distinguished curve of orthogonality along which the asymptotic behavior of $p_n^\lambda$ changes depending on whether $z$ belongs to this curve or not. This curve should be the one where the zeros of $p_n^\lambda$ asymptotically lie, as depicted in Figure~\ref{fig: intro pic}.
	
	The rigorous study of this intuitive notion of the \say{correct} curve to choose from was initiated by Nuttall, who conjectured in a non-varying situation that such correct curve should be of minimal capacity (see for instance the later work of Nuttall and Singh \cite{nuttall_singh}). Nutall's conjecture was later established rigorously by Stahl in a series of seminal works \cite{stahl1985extremal,stahl1986orthogonal}. Roughly, the idea is that for a fixed allowable curve of orthogonality $\gamma$, one first minimizes an energy functional over probability measures defined on $\gamma$, and then maximizes this minimum over all possible $\gamma$. The curve $\gamma=\gamma_S$ for which this max-min is attained, known after the works of Stahl by the name of S-curve, is then the one that attracts the zeros of the corresponding large degree orthogonal polynomials. Moreover, the probability measure that minimizes the energy functional on $\gamma_S$ governs the limiting distribution of the zeros. For the interested reader, we also recommend the more recent work \cite{aptekarevyattselev2015} by Aptekarev and Yattselev, where they revisit Nuttall's conjecture under a more modern perspective, establishing strong asymptotics of non-hermitian orthogonal polynomials under some additional geometric conditions.
	
	The original works of Stahl mentioned above were concerned with non-hermitian orthogonality with non-varying weights. In an attempt to extend Stahl's results to account for varying weights, Gonchar and Rakhmanov \cite{gonchar1989equilibrium} obtained the asymptotic zero distribution of a certain class of non-hermitian orthogonal polynomials with varying weights, but took the existence of the associated S-contour for granted. The problem of existence of the S-contour, not only in the context of Gonchar and Rakhmanov's original work but also in other contexts, remained open until not so long ago, when Rakhmanov \cite{rakhmanov2012orthogonal} outlined a very general max-min approach for obtaining S-contours. The rigorous analysis of this approach, nowadays called the Gonchar-Rakhmanov-Stahl program, depends heavily on the type of orthogonality weight and the geometry of the contours at hand, but nevertheless has been carried out in various different settings \cite{aptekarevyattselev2015,kuijlaars_silva_2015,martinezrakhmanov2011,martinez2016orthogonal,yattselev2018}, many of which were largely inspired by Rakhmanov's outline in \cite{rakhmanov2012orthogonal}.
	
	Despite the many works on S-contours already available in the literature, a rigorous proof of the existence of such a contour adapted to the context of the orthogonality as in \eqref{eq: varying weight kissing polynomials} and arbitrary polynomial potential $V$ remains an open question. By virtue of the simplicity of our choice of potential $V$ in \eqref{eq: varying weight kissing polynomials}, instead of trying to obtain the S-contour rigorously through a max-min approach, we use ad hoc calculations inspired by these max-min techniques to find an educated guess for the appropriate form of our S-contour. The main output of these calculations is an expression for the so-called spectral curve, from which we construct a $g$-function that should ultimately describe the leading exponential asymptotics of the polynomials of large degree. Moving forward, we use this $g$-function to implement the Riemann-Hilbert approach that finally provides the rigorous asymptotics of the orthogonal polynomials.
	
Much in virtue of the connection with random matrices and other mathematical physics' models, asymptotics of orthogonal polynomials, both hermitian but also non-hermitian, have experimented an emerging of new ideas and techniques. In contrast to the potential-theoretic techniques by Stahl, Gonchar and Rakhmanov, among others, just described, another possible approach to obtain the $g$-function that governs the leading asymptotics of orthogonal polynomials is through deformations techniques. In short terms, the main idea is to see this $g$-function as being deformed with the parameter $\lambda$, and keep track of this deformation in the parameter space $\lambda$. These ideas have shown to be quite powerful in the context of integrable systems for a little while, see for instance the monograph by Kamvissis, McLaughlin  and Miller \cite{kamvissis_mclaughlin_miller} and the references therein, but have more recently also come to play in the theory of orthogonal polynomials, notably by Bertola and Tovbis \cite{bertola2011boutroux,bertola_tovbis_2015,bertola_tovbis_2016}. 

In contrast with the potential-theoretic approach, where one constructs the $g$-function with the solution of a variational problem, over here one finds a scalar Riemann-Hilbert problem for the $g$-function (or yet its derivative), where the number of ``cuts'' (or the genus of the underlying Riemann surface) is also to be determined. Nevertheless, once one proves that for a given number of cuts and a specific parameter value $\lambda=\lambda_0$ this scalar Riemann-Hilbert problem yields the correct $g$-function, one then has to deform $\lambda$, keeping track of all the level lines for $\re g$. The critical transition in the parameter space $\lambda$ then appears as a change in behavior of these level lines, and past those critical points one possibly has to adapt the number of cuts one started with. In the recent work \cite{bertola_tovbis_2016}, for instance, Bertola and Tovbis carry out this approach for a quartic potential, but many of their technical results are actually valid for general potential when the contour of orthogonality does not have finite endpoints. Such approach could in principle be adapted to our context here, where we have the endpoints $\pm 1$, as well. In fact, many of the aspects of the present work are comparable, for instance the need of control for the level lines of $\re g$ is, to some extent, similar with our analysis of trajectories of quadratic differentials in Section~\ref{sec: trajectories} below, and \eqref{eq: spectral curve 1} below is equivalent to the Riemann-Hilbert problem for the derivative of the $g$-function. 

Under the perspective of the matrix Riemann-Hilbert approach for the orthogonal polynomials, either the potential-theoretic approach or the deformations in the parameter space technique share in common that they do not provide the strong asymptotics of the orthogonal polynomials directly. Instead, they do provide the key object, the $g$-function, that enters as an input of the Riemann-Hilbert analysis and with which one is then able to apply the Deift-Zhou steepest descent method.
	
	Following along the lines of the max-min approach used in \cite{kuijlaars_silva_2015,martinezrakhmanov2011}, see also \cite{deano2014large} for related calculations, we expect that the weak limit of the normalized zero counting measure for $p_n^\lambda$, say a measure $\mu_*$, should verify a quadratic equation of the form
	\begin{equation}\label{eq: spectral curve 1}
	\left(\int \frac{d\mu_*(s)}{s-z} + \frac{V'(z)}{2}\right)^2=Q(z),\quad z\in \C\setminus \supp\mu_*,
	\end{equation}
	where $Q$ is a rational function to be determined, whose only singularities are simple poles at $\pm 1$ (so as to encode that the endpoints of integration in \eqref{eq: varying weight kissing polynomials} are $\pm1$). A comparison of both sides of this quadratic equation implies that
	\begin{equation}\label{eq: branch square root}
	Q^{1/2}(z) = -\frac{i\lambda}{2} - \frac{1}{z} + \mathcal{O}\left(\frac{1}{z^2}\right), \qquad z \to \infty, 
	\end{equation}
	so $Q$ should take the form
	\begin{equation}\label{def: rational Q}
	Q(z)=Q(z;\lambda,x)=-\frac{\lambda^2}{4}\frac{(z-z_\lambda(x))(z+\overline{z_\lambda(x)})}{z^2-1}, 
	\end{equation}	
	for
	\begin{equation}\label{def: z lambda}
	z_\lambda(x) = x + \frac{2i}{\lambda},
	\end{equation}
	and $x \in \mathbb{R}$ yet to be chosen appropriately.
	
	For the choice $x=0$ the polynomial $Q$ has a double zero at $z_\lambda(0)$, and the Riemann surface associated to the equation
	\begin{equation}\label{eq: spectral curve}
	\xi^2=Q(z),
	\end{equation}
	for which 
	$$
	\xi=\xi_1(z)=\int \frac{d\mu_*(s)}{s-z} + \frac{V'(z)}{2}
	$$
	should be a solution to, has genus $0$. This genus ansatz yields the correct guess of an appropriate $Q$ for $\lambda<\lambda_c$, and it is consistent with the numerical observation that for $\lambda<\lambda_c$, the zeros of $p_n^\lambda$ accumulate on a single analytic arc, as proven in the aforementioned work \cite{deano2014large}. In the same work, Deaño also indicated that $x=0$ should not be the correct choice for $\lambda>\lambda_c$. In light of the numerical outputs in Figure~\ref{fig: intro pic}, which indicate that for $\lambda>\lambda_c$ the zeros accumulate on two disjoint arcs, we actually expect that the Riemann surface associated to $Q$ as in \eqref{eq: spectral curve} has genus $1$. This means that $Q$ must have simple zeros for $\lambda>\lambda_c$, so we must expect that the correct choice $x=x_*$ must satisfy $x_*>0$.
	
	The determination of $x_*$ in the supercritical regime is our first contribution.
	
	\begin{theorem}\label{thm: boutroux condition}
		Suppose that $\lambda>\lambda_c$ and define $z_\lambda$ as in \eqref{def: z lambda}. Then there exists a unique choice $x=x_*=x_*(\lambda)\in (0,1)$ for which
		\begin{equation}\label{eq: boutroux condition}
		\re \int_{z_\lambda(x_*)}^{1}\sqrt{Q(s)}ds=0.
		\end{equation}
		Furthermore,
		$$
		\lim_{\lambda\to +\infty} x_*(\lambda) =1.
		$$
	\end{theorem}

With some effort, the limit above could be complemented with the limit
$$
\lim_{\lambda\to\lambda_c^+} x_*(\lambda)=0,
$$	
see for instance Remark~\ref{rmk:uniqueness_x} below.
	
In the literature, the transcendental condition \eqref{eq: boutroux condition} goes by the name of the {\it Boutroux condition} \cite{bertola2011boutroux,bertola2009commuting,bertola_tovbis_2015,kuijlaars_tovbis_2015}. We remark that this condition does not depend on the choice of branch of the square root, as long as it varies analytically along the contour of integration. We are implicitly assuming this fact when we write \eqref{eq: boutroux condition}. In a moment we will fix a branch of this root that will be used throughout the rest of this introduction.
	
	Once the function $Q$ in \eqref{eq: spectral curve 1} is determined, if we knew $\supp\mu_*$, then we could simply use Plemelj's formula in \eqref{eq: spectral curve 1} to recover the density of $\mu_*$. Even if we did not know $\supp\mu_*$, this observation would still be useful, as it imposes a constraint on what $\supp\mu_*$ could potentially look like. Indeed, assuming that $\supp\mu_*$ is a union of arcs, using Plemelj's formula in \eqref{eq: spectral curve 1}, we see that $\mu_*$ should verify the relation $\pi i\; d\mu_*(s)=Q^{1/2}_+ds$, where $ds$ is the complex line element along the arcs of $\supp\mu_*$. In particular, this formal calculation indicates that $Q_+^{1/2}ds$ has to be purely imaginary along the support of $\mu_*$, therefore imposing a severe restriction on what types of arcs are good candidates for $\supp\mu_*$.
	
	In fact, our next two results transform these ideas into rigorous statements. We first provide the existence of the analytic arcs that will support the limiting zero distribution of $p_n^\lambda$. 
	
	\begin{theorem}\label{thm: support zero distribution}
		Take $Q$ and $z_*:=z_\lambda(x_*)$ as in \eqref{def: rational Q} and \eqref{def: z lambda}, where $x_*$ is defined as in Theorem~\ref{thm: boutroux condition}. Then there exist analytic arcs $\gamma_1$ and $\gamma_2$ with the following properties.
		\begin{enumerate}[(i)]
			\item The arc $\gamma_2$ is on the right half plane, starts at $z_*$ and ends at $1$, and it is the unique such arc that satisfies
			\begin{equation}\label{eq: pretrajectory1}
			\int_{z_*}^z \sqrt{Q(s)}ds\in i\R,\quad z\in \gamma_2.
			\end{equation}
			\item The arc $\gamma_1$ is obtained as the reflection of $\gamma_2$ over the imaginary axis and satisfies
			\begin{equation}\label{eq: pretrajectory2}
			\int_{-1}^z \sqrt{Q(s)}ds\in i\R,\quad z\in \gamma_1.
			\end{equation}
		\end{enumerate}
	\end{theorem}	
	
	In the same spirit as \eqref{eq: boutroux condition}, the conditions \eqref{eq: pretrajectory1}--\eqref{eq: pretrajectory2} do not depend on the choice of the branch of the root, as long as this choice varies analytically along the contour of integration. But now it is a good time to define, once and for all, the branch of the root that we will be dealing with for the rest of this section.
	
	To do so, we orient the arcs $\gamma_1$ and $\gamma_2$ from $-1$ to $-\overline z_*$ and from $z_*$ to $1$, respectively, and for convenience set $\gamma=\gamma_1\cup\gamma_2$. With the orientation inherited from $\gamma_1$ and $\gamma_2$, the arc $\gamma$ has natural $\pm$-sides. The rational function $Q$ has a well-defined analytic square root on $\C\setminus \gamma$, that we choose in such a way that the asymptotic expansion \eqref{eq: branch square root} holds true. For $z\in \gamma$, we denote by $Q^{1/2}_\pm(z)$ the boundary values of this square root of $Q$ when we approach $\gamma$ from its $\pm$-side.
	
	The next result assures the existence of a positive measure $\mu_*$, whose Cauchy transform solves \eqref{eq: spectral curve}. This measure, as formally stated in a moment, will turn out to be the limiting zero distribution of the kissing polynomials with varying weight.
	
	\begin{theorem}\label{thm: density zero distribution}
		Suppose $\lambda>\lambda_c$ and let $Q$ be defined as in \eqref{def: rational Q} with the choice of $x_*$ given by Theorem~\ref{thm: boutroux condition}. Define a complex-valued measure $\mu_*$ on $\gamma$ through its density w.r.t. the complex line element $ds$ as
		$$
		d\mu_*(s)=\frac{1}{\pi i}Q^{1/2}_+(s)ds,\quad s\in \gamma.
		$$
		Then $\mu_*$ is, in fact, a probability measure on $\gamma$, and its shifted Cauchy transform 
		$$
		\xi_1(z)=C^{\mu_*}(z)- \frac{i\lambda}{2},\quad C^{\mu_*}(z):=\int \frac{d\mu_*(s)}{s-z},\quad  z\in \C\setminus \gamma,
		$$
		solves \eqref{eq: spectral curve} for $z\in \C\setminus \gamma$,
	\end{theorem}
	
	With the measure $\mu_*$ in hand, we use it as an input for the asymptotic analysis of $p_n^\lambda$ for large $n$. This asymptotic analysis is based on the Riemann-Hilbert formulation for orthogonal polynomials which, combined with the Deift-Zhou nonlinear steepest descent method, provides not only the weak zero distribution for $p_n^{\lambda}$, but also strong asymptotic formulas for $p_n^\lambda$. 
	
	To discuss some of these asymptotic results, introduce the quantity
	\begin{equation}\label{def: constant kappa}
	\kappa=\kappa(\lambda):=-i\int_{-\overline z_*}^{z_*} Q^{1/2}(s)\,ds=\im\int_{-\overline z_*}^{z_*} Q^{1/2}(s)\,ds.
	\end{equation}
	where the path of integration is a straight line segment and the last identity follows from symmetry considerations. This real value $\kappa$ is a real-analytic function of $\lambda>\lambda_c$. As we will show in Section~\ref{sec: global parametrix}, for $n$ odd and a function $c=c(\lambda)$ which will be defined later with the help of \eqref{def: theta divisor}--\eqref{def: meromorphic differential n k}, the difference $2n\kappa(\lambda)-c(\lambda)$ takes values in $2\pi \Z$ only for a discrete set of values of $\lambda$  and $n$. With this in mind, for $\varepsilon>0$, we define the set $\Theta_\varepsilon^*$ to be the set of pairs $(n,\lambda)$ for which the quantity $2n\kappa(\lambda)-c(\lambda)$ is within a distance $\epsilon$ from $2\pi \Z$. More details are given in Section~\ref{section: analysis theta divisor}. Set
	\begin{equation}\label{eq: phi 4 eqn}
	\phi(z) = \int_{1}^{z} Q^{1/2}(s)\, ds + \frac{i\kappa}{2},\quad z\in \C\setminus (\gamma_1\cup\gamma_2).
	\end{equation}
	As we will verify later, the function $\phi$ is well defined modulo $\pi i$. Also, from the expansion \eqref{eq: branch square root} we see that
	\begin{equation}\label{eq: expansion phi}
	\phi(z) = - \frac{i\lambda z}{2} - \log z +\frac{i\kappa}{2} - l + \mathcal{O}\left(\frac{1}{z}\right), \qquad z \to \infty.
	\end{equation}
	for some complex constant $l$.
	
	\begin{theorem}\label{thm: strong asymptotics gamma}
		Fix $\varepsilon>0$ and suppose $\lambda>\lambda_c$. 
		For $n$ sufficiently large, and $(n,\lambda)\not\in \Theta^*_\varepsilon$ in case $n$ is odd, the kissing polynomial $p_n^\lambda$ in \eqref{eq: varying weight kissing polynomials} uniquely exists as a monic polynomial of degree exactly $n$, and the weak asymptotics of its zeros $z_1,\hdots,z_n$ is given by the weak limit
		\begin{equation}\label{eq: weak convergence zeros}
		\lim_{n\to \infty} \frac{1}{n}\sum_{k=1}^n \delta_{z_k}\stackrel{*}{=}\mu_*,
		\end{equation}
		where $\delta_z$ is the atomic measure with mass $1$ at $z$.
		
		Furthermore, as $n\to \infty$ the asymptotic formulas
		\begin{equation}\label{eq: asymptotics outside gamma}
		\begin{aligned}
		p^\lambda_{2n}(z) & =\Psi_{n,0}(z)e^{2n\left(\frac{i\kappa}{2}-\frac{i \lambda}{2}z-l-\phi(z)\right)}\left(1+\Boh(n^{-1})\right),\\
		p^\lambda_{2n+1}(z) & =e^{(2n+1)\left(\frac{i\kappa}{2}-\frac{i \lambda}{2}z-l-\phi(z)\right)}\left(\Psi_{n,1}(z)+\Boh(n^{-1})\right) 
		\end{aligned}
		\end{equation}
		hold true uniformly in compacts of $ \C\setminus (\gamma_1\cup \gamma_2)$ and $\C\setminus (\gamma_1\cup \widehat{ \gamma}\cup \gamma_2)$, respectively, where the functions $\Psi_{n,0}$ and $\Psi_{n,1}$ have the following properties.
		\begin{enumerate}[(i)]
			\item $\Psi_{n,0}$ is holomorphic on $\C\setminus (\gamma_1\cup\gamma_2)$, whereas $\Psi_{n,1}$ is holomorphic on $\C\setminus (\gamma_1\cup \widehat{ \gamma}\cup \gamma_2)$, where $\widehat \gamma$ is a contour connecting $-\overline z_*$ and $z_*$, and they remain bounded in compacts of their domains of definition as $n\to \infty$.
			\item $\Psi_{n,0}$ does not have zeros.
			\item The function $\Psi_{n,1}$ has a unique zero at a point $a_*=a_*(n,\lambda)$, which is simple and located on the imaginary axis.
		\end{enumerate}
	\end{theorem}
	
	We note that although $\Psi_{n,1}$ has a jump on a contour $\widehat \gamma$ which does not contain zeros of $p^{\lambda}_{2n+1}$, it actually turns out that $\Psi_{n,1}e^{-(2n+1)\phi}$ is analytic on $\C\setminus (\gamma_1\cup\gamma_2)$, so the leading term on the right-hand side of \eqref{eq: asymptotics outside gamma} is in fact analytic on $\C\setminus (\gamma_1\cup\gamma_2)$.
	
	The nature of the restriction on odd $n$ in Theorem~\ref{thm: strong asymptotics gamma} is due to the construction of the so-called global parametrix, whose existence can only be assured upon verifying the non-degeneracy conditions leading to the definition of $\Theta_\varepsilon^*$ above. Again, details are given in Section~\ref{sec: global parametrix}, and in particular Section~\ref{section: analysis theta divisor}. The functions $\Psi_{n,0}$ and $\Psi_{n,1}$ are specific entries in this global parametrix, and as such they are initially constructed with the help of meromorphic differentials on the Riemann surface associated to \eqref{eq: spectral curve}. For the benefit of the reader, in Section~\ref{sec: asymptotics} we translate this construction into complex-analytic terms, which leads to more explicit forms for $\Psi_{n,0}$ and $\Psi_{n,1}$.
	
	The appearance of a zero of $\Psi_{n,1}$ on the imaginary axis is natural in our situation, because for odd $n$ the polynomial $p_n^\lambda$ has, by symmetry, exactly one zero on the imaginary axis. When $\lambda<\lambda_c$, the support of the limiting zero distribution of $p_n^{\lambda}$ is connected and intersects the imaginary axis, so this zero on $i\R$ is always encoded in the limiting distribution. However, when $\lambda>\lambda_c$, the support of the limiting distribution $\mu_*$ no longer touches the imaginary axis. Therefore, the purely imaginary zero of $p_{2n+1}^\lambda$ remains an outlier, and can intuitively be thought of as a \say{spurious} zero in the language of rational approximation. 
	
	\begin{figure}[t]
		\centering
		\begin{subfigure}[b]{0.32\linewidth}
			\includegraphics[width=\linewidth]{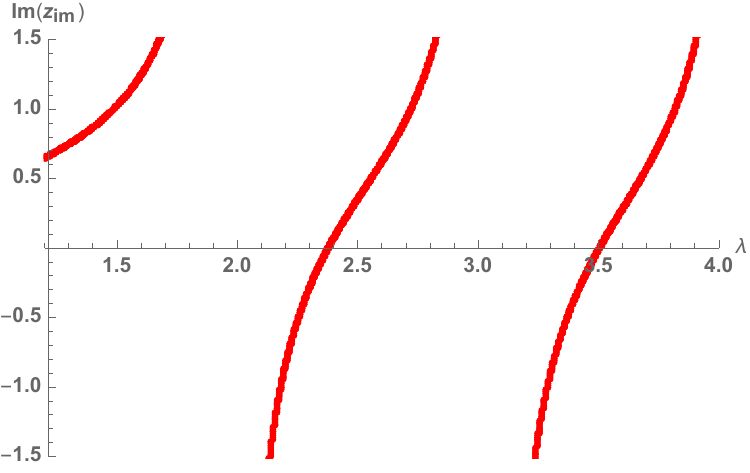}
			\caption*{$n= 1$}
			\label{fig: spur3}
		\end{subfigure}
		\begin{subfigure}[b]{0.32\linewidth}
			\includegraphics[width=\linewidth]{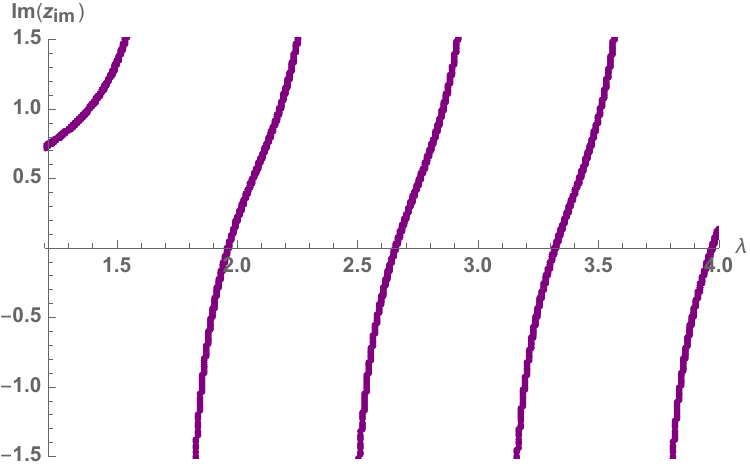}
			\caption*{$n = 2$}
			\label{fig: spur5}
		\end{subfigure}
		\begin{subfigure}[b]{0.32\linewidth}
			\includegraphics[width=\linewidth]{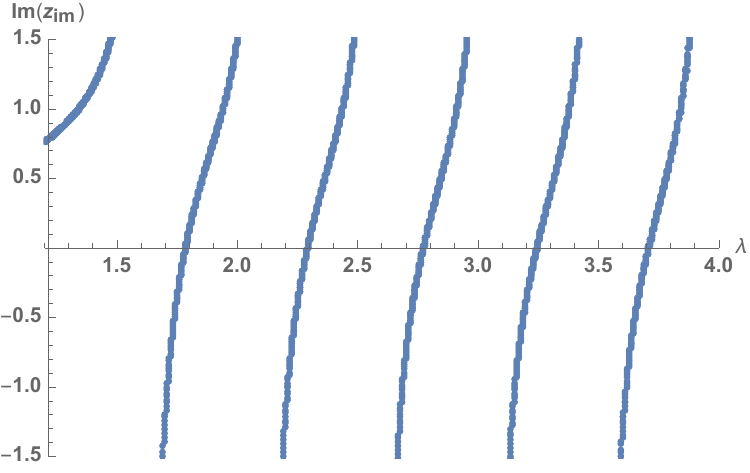}
			\caption*{$n = 3$}
			\label{fig: spur7}
		\end{subfigure}
		\caption{Plotting the imaginary part of the purely imaginary zero of $p_{2n+1}^\lambda$ as a function of $\lambda$, for $n=1,2,3$.}
		\label{fig: spur}
	\end{figure}
	
	As it was pointed out by Deaño, Huybrechs and Iserles \cite[Proposition~3]{deano2015kissing}, for any $n$ fixed, there is always a sequence $\omega=\omega_j\to \infty$ for which $p_{2n+1}^{\omega_j}$ never exists (as a polynomial of degree exactly $2n+1$). Having in mind the identification $\omega=n\lambda$, and leaving aside technicalities such as the uniformity of the error in \eqref{eq: asymptotics outside gamma} for large $\lambda$, it is therefore natural that $p_{2n+1}^\lambda$ need not exist for all values of $\lambda$. As such, the restriction on odd degrees in Theorem~\ref{thm: strong asymptotics gamma} was to be expected.

	\subsection*{Overview of the Paper} The present work is structured as follows. 
	
	In Section \ref{sec: boutroux} we prove Theorem~\ref{thm: boutroux condition}. The proof is based on the symmetries of the rational function $Q$, as well as on some explicit conformal mapping calculations.
	
	Theorems~\ref{thm: support zero distribution} and \ref{thm: density zero distribution} are proven in Section~\ref{sec: trajectories}. These proofs are based on the theory of quadratic differentials, and are greatly inspired by ideas developed by Mart\'inez-Finkelshtein and the second-named author, together with collaborators \cite{martinez_finkelshtein_silva2016,martinez2016complexjacobi,martinez2014complexlaguerre,faouzi2016complexjacobi,bleher_silva}.
	
	The asymptotic analysis of $p_n^\lambda$ is performed via the Riemann-Hilbert formulation for the orthogonal polynomials combined with the Deift-Zhou nonlinear steepest descent method \cite{deift1999strong,deift1999uniform}. Despite the fact that we are dealing with complex orthogonality, much of the analysis follows in parallel with the case of real orthogonality \cite{kuijlaars2004riemann}. In Section \ref{sec: rhp}, we make a brief overview of this method, quickly describing the outcome of its first steps. Although this asymptotic analysis, at the algebraic level, does not differ much from previous analysis in the literature, it is worth mentioning that a crucial object in the first steps is the rational function $Q$ obtained through Theorem~\ref{thm: boutroux condition}, or equivalently, the corresponding measure $\mu_*$ in Theorem~\ref{thm: density zero distribution}.
	
 Different constructions of the so-called global parametrix are present in the literature, for instance in the weighted case on higher genera Riemann surfaces with real symmetry \cite{deift1999uniform}, in the unweighted case without symmetries \cite{aptekarevyattselev2015,yattselev2015} or yet in the logarithmically weighted case without symmetries as well \cite{yattselev2018}; all these works rely on constructions using theta functions. In our work, the main substantial difference in the asymptotic analysis appears in the construction of this parametrix. Due to the lack of real symmetry to explore, certain technical obstacles arise when trying to solve the associated model Riemann-Hilbert problem. This was already evidenced in the restriction on odd degrees in the statement of Theorem~\ref{thm: strong asymptotics gamma} and briefly discussed thereafter. At any rate, we follow ideas similar to the ones developed by Kuijlaars and Mo \cite{kuijlaars2011global}, but provide calculations in a more explicit manner so as to be able to construct this parametrix, rather than simply assure its existence. This construction is carried out in Section~\ref{sec: global parametrix}. The restriction $(n,\lambda)\not\in \Theta^*_\varepsilon$ in Theorem~\ref{thm: strong asymptotics gamma} also appears for the first time in this section, and is thoroughly discussed. 
	
	In Section \ref{sec: conclusion steepest descent} we briefly discuss the construction of local parametrix, which in our situation turns out to be constructed with the use of Bessel (for the endpoints $\pm 1$) and Airy (for the endpoints $z_*$ and $-\overline z_*$) functions. At the end of this section, we conclude the steepest descent analysis.
	
	Finally, in Section~\ref{sec: asymptotics} we wrap up the conclusions drawn from the steepest descent analysis, hence obtaining asymptotic formulas for $p_n^\lambda$, and in particular concluding the proof of the asymptotic results in Theorem~\ref{thm: strong asymptotics gamma}.
	
	\subsection*{Acknowledgments} 
	The authors thank Alfredo Deaño, Arieh Iserles and Andrei Martínez-Finkelshtein for very fruitful discussions. They are also grateful to Ahmad Barhoumi and Maxim Yattselev, who pointed out a mistake in a previous version of Theorem 2.4. This collaboration started when the authors attended the conference {\it Foundations of Computational Mathematics 2017} (FOCM 2017), and the authors are also grateful to the organizers of this conference for providing an excellent atmosphere for discussion. Most of this work was carried out while G.S. was a Postdoctoral Assistant Professor at the University of Michigan. He also acknowledges his current support by São Paulo Research Foundation under grant \#2019/16062-1.
	
\section{The Boutroux Condition}\label{sec: boutroux}
	The Boutroux Condition, introduced in \cite{bertola2011boutroux,bertola2009commuting}, provides one approach to determining the asymptotics of orthogonal polynomials with respect to complex weights. 
	For a rational function $R$, the Boutroux condition asks that
	\begin{equation*}
	\oint_\alpha \xi \, ds \in i \mathbb{R},  
	\end{equation*}
	where $\alpha$ is any closed loop on the Riemann surface associated to the algebraic curve $\xi^2=R(z)$. 
	
	In the present setting, for any $x>0$, fix $\lambda>\lambda_c$ and consider
	\begin{equation}
		z_\lambda(x) = x+\frac{2i}{\lambda}, \qquad x>0,
	\end{equation} 
	and the associated rational function $Q=Q(\cdot;\lambda,x)$ as in \eqref{def: rational Q}. 
	Let $L$ be the union of the oriented line segments connecting $-1$ to $-\overline{z_\lambda(x)}$ and $z_\lambda(x)$ to $1$. In this section we use $L$ as a set of branch cuts for $Q^{1/2}$, so that we fix $Q^{1/2}$ to be the square root of $Q$ which is analytic on $\C\setminus L$ and that satisfies the asymptotics in \eqref{eq: branch square root}. Notice that the endpoints of $L$ (and, in loose terms, $L$ itself) vary continuously with the parameters $x\geq 0$ and $\lambda>0$.
	
	In very concrete terms, the Boutroux condition in our setting asks to find $x$ for which
	\begin{equation}\label{eq: periods}
	\re \int_{-1}^{-\overline{z_\lambda(x)}} Q_+^{1/2}(s)\, ds=\re  \int_{-\overline{z_\lambda(x)}}^{z_\lambda(x)} Q^{1/2}(s)\, ds =  \re\int_{z_\lambda(x)}^{1} Q_+^{1/2}(s)\, ds =0,
	\end{equation}
	where the integration takes place along straight line segments (that is, along subarcs of $L$), and we recall that the subscript $+$ denotes the limiting value of $Q^{1/2}\left(s\right)$ as we approach $s\in L$ from its left-hand side w.r.t. the orientation of $L$. We also point out to the reader that the last integral in \eqref{eq: periods} is exactly the same as \eqref{eq: boutroux condition}.
	
	\begin{prop}\label{prop: symmetry integrals}
		\begin{equation*}
		-\overline{\int_{z_\lambda(x)}^{1}Q_+^{1/2}(s)\; ds}=\int_{-1}^{-\overline z_\lambda(x)}Q_+^{1/2}(s)\; ds \quad \mbox{and} \quad \re \int_{-\overline{z_\lambda(x)}}^{z_\lambda(x)} Q_+^{1/2}(s)\, ds=0.
		\end{equation*}
	\end{prop}
	\begin{proof}
		Both identities follow from the symmetry $Q^{1/2}(\overline z)=-\overline {Q^{1/2}(-z)}$, which is valid for $z$ outside the contours of integration and extends to the cuts as $Q_+^{1/2}(\overline z)=-\overline {Q_+^{1/2}(-z)}$. See also Appendix~\ref{appendix: symmetries}.
	\end{proof}
	
	As a result of Proposition~\ref{prop: symmetry integrals}, to determine $x$ such that equations \eqref{eq: periods} are satisfied, all we have to do is to make sure that the last integral in \eqref{eq: periods} is purely imaginary.  To do this, we consider the function
	\begin{equation}\label{eq: psi lambda defn}
	\psi(x) =\re \int_{z_\lambda(x)}^1 Q_+^{1/2}(s)\, ds= \re \int_{x+\frac{2i}{\lambda}}^1 Q_+^{1/2}(s;\lambda, x)\, ds.
	\end{equation}
	We emphasize that for any fixed $x$, the contours of integration are still straight line segments. This assures that for fixed $\lambda$, the function $\psi$ is a continuous function of $x\in [0,1]$. Our next task is to show that $\psi$ changes sign on $[0,1]$. 
	\begin{prop}\label{prop: negative}
		For any $\lambda > \lambda_c$, we have that
		\begin{equation*}
		\psi(0) < 0.
		\end{equation*}
	\end{prop}
	\begin{proof}
		We have that 
		\begin{equation*}
		\psi\left(0\right) = \text{Re }\int_{\frac{2i}{\lambda}}^{1} Q_+^{1/2}\left(s;\lambda, 0\right)\,ds.
		\end{equation*}
		In this situation, we can write $Q^{1/2}$ explicitly as
		\begin{equation*}
		Q^{1/2}(s;\lambda,0) = \frac{-i\lambda}{2}\frac{(s-\frac{2i}{\lambda})}{\sqrt{s^2-1}},
		\end{equation*}
		where the branch of $\sqrt{s^2-1}$ is the principal branch, so that $\sqrt{s^2-1} \sim s$, as $s\to \infty$. Using that
		$$
		\frac{d}{ds}\left(\sqrt{s^2-1} \right)=\frac{s}{\sqrt{s^2-1}}\quad \mbox{and} \quad \frac{d}{ds}\left(\frac{1}{2}\log\left(\frac{\sqrt{s^2-1}+s}{\sqrt{s^2-1}-s}\right)\right)=\frac{1}{\sqrt{s^2-1}},
		$$
		a cumbersome, but straightforward calculation gives us that
		\begin{equation*}
		-\frac{i\lambda}{2} \int_{\frac{2i}{\lambda}}^1 \frac{(s-\frac{2i}{\lambda})}{\sqrt{s^2-1}}\, ds = 	-\frac{\sqrt{4+\lambda^2}}{2}+\log\left(\frac{i\left(2+\sqrt{4+\lambda^2}\right)}{\lambda}\right),
		\end{equation*}
		so that by taking real parts, we have that
		\begin{equation*}
		\psi(0) = \log\left(\frac{2+\sqrt{4+\lambda^2}}{\lambda}\right)-\frac{\sqrt{4+\lambda^2}}{2}.
		\end{equation*}
		Note that $\psi(0)=0$ when $\lambda=\lambda_c$, which follows from the definition of $\lambda_c$ as the only positive solution to \eqref{eq: lambda crit defn}. Furthermore,
		\begin{equation*}
		\frac{d}{d\lambda} \psi(0) = -\frac{\sqrt{4+\lambda^2}}{2\lambda} <0,
		\end{equation*}
		so $\psi(0) < 0$ for all $\lambda > \lambda_c$, as desired. 
	\end{proof}
	
	
	\begin{prop}\label{prop: psi lambda 1 > 0}
		For all $\lambda >  0$, we have that
		\begin{equation*}
		\psi(1) > 0.
		\end{equation*}
	\end{prop}	
	\begin{proof}
		Through the linear change of variables
		$$
		s\mapsto i(s-1)
		$$
		we see that
		\begin{equation}\label{eq: translated rotated psi 1}
		\psi\left(1\right) = -\frac{\lambda}{2}\int_{-\frac{2}{\lambda}}^0 \text{Re }R^{1/2}_+(s)\, ds,
		\end{equation}
		where
		\begin{equation*}
		R(s) = \frac{\left(2+s\lambda\right)\left(2+\lambda(2i+s)\right)}{\lambda^2 s \left(2i+s\right)}.
		\end{equation*}
		As we have fixed the branch of $Q^{1/2}$, the branch of $R^{1/2}$ in \eqref{eq: translated rotated psi 1} is the one that behaves like
		\begin{equation*}
		R^{1/2}(s) \to 1, \qquad s \to \infty.
		\end{equation*}
		Here, $R^{1/2}$ has branch cuts on the horizontal segments $\left(-\frac{2}{\lambda},0\right)$ and $\left(-\frac{2}{\lambda}-2i,-2i\right)$ and the integral \eqref{eq: translated rotated psi 1} is computed along the first of these branch cuts. The goal now is to show that
		\begin{equation*}
		\text{Re } R_+^{1/2}(s) < 0, \qquad s\in \left(-\frac{2}{\lambda},0\right),
		\end{equation*}
		which will immediately imply that $\psi(1)>0$ for $\lambda > 0$. To do this, first note that for $s\in \R$ we can split $R(s)$ into real and imaginary parts as 
		\begin{equation*}
		R(s) = U(s) + i V(s), \quad	U(s) = \frac{\left(2+s\lambda\right)\left(2s+\lambda\left(4+s^2\right)\right)}{\lambda^2 s\left(4+s^2\right)},
		\quad
		V(s) = -\frac{4\left(2+s\lambda\right)}{\lambda^2\left(s^3+4s\right)}.
		\end{equation*}
		As $s \to -\infty$, 
		\begin{equation*}
		R(s) = 1 + \frac{4}{\lambda s} + \frac{4-4i\lambda}{\lambda^2 s^2} + \mathcal{O}\left(\frac{1}{s^3}\right),
		\end{equation*}
		so that as $s$ moves from $-\infty$ towards $-2/\lambda$ on the negative real axis, the image $R(s)$ traces out a curve in the plane, starting at $z=1$ and initially dropping into the lower-right hand quadrant. As neither the real nor imaginary parts of $R(s)$ have real zeros or poles in the interval $\left(-\infty, -\frac{2}{\lambda}\right)$, we can conclude that $R(s)$ remains in this quadrant for these values of $s$, and consequently
		\begin{equation*}
		\arg R^{1/2}(s) \in \left(-\frac{\pi}{4},0\right), \qquad s\in \left(-\infty,-\frac{2}{\lambda}\right).
		\end{equation*}
		In particular, from the expansion
		\begin{equation*}
		R(s) =-\frac{\lambda^2}{1+\lambda^2}(\lambda-i)\left(s+\frac{2}{\lambda}\right) + \mathcal{O}\left(\left(s+\frac{2}{\lambda}\right)^2\right), \qquad s \to -\frac{2}{\lambda},
		\end{equation*}
		we obtain
		\begin{equation*}
		\lim\limits_{\substack{s\to -\frac{2}{\lambda}\\s<-\frac{2}{\lambda}}} \arg R^{1/2}(s)=\lim\limits_{\substack{s\to -\frac{2}{\lambda}\\s<-\frac{2}{\lambda}}} \frac{1}{2}\arctan\frac{V(s)}{U(s)} = - \frac{\arctan \left(\frac{1}{\lambda}\right)}{2}\in \left[-\frac{\pi}{4},0\right].
		\end{equation*}
		Because $R^{1/2}$ vanishes as a square root at $s=-2/\lambda$, a conformal mapping analysis then implies that
		\begin{equation*}
		\lim\limits_{\substack{s\to -\frac{2}{\lambda}\\s>-\frac{2}{\lambda}}} \arg R_+^{1/2}(s) =- \frac{\arctan 	\left(\frac{1}{\lambda}\right)}{2}-\frac{\pi}{2} \in \left(-\frac{3\pi}{4},-\frac{\pi}{2}\right),
		\end{equation*}
		so that as $s$ moves from $s=-\frac{2}{\lambda}$ towards $0$ on the negative real axis, the image of $R^{1/2}_+(s)$ traces out a curve that starts at $0$ and ventures into the lower left hand quadrant of the plane. Just as before, neither the real nor imaginary parts of $R(s)$ have zeros or poles in the interval $\left(-\frac{2}{\lambda},0\right)$, and therefore we obtain
		\begin{equation*}
		\re R^{1/2}_+(s) < 0, \qquad s\in \left(-\frac{2}{\lambda},0\right),
		\end{equation*}
		as desired. 
	\end{proof}
	
	We are ready to prove the main result of this section.
	
	\begin{proof}[Proof of Theorem~\ref{thm: boutroux condition}]
		The existence of $x_*\in (0,1)$ follows from Proposition~\ref{prop: negative} and \ref{prop: psi lambda 1 > 0} and the continuity of $\psi(x)$. With existence covered, we define $z_* := z_\lambda(x_*)$.
		
		The uniqueness of such $x_*$ will follow later, in a more indirect manner. As this is not an essential part in any of the coming, we only outline the proof of this statement in the following steps:
		\begin{enumerate}[$\bullet$]
			\item The whole asymptotic analysis to be done later relies only on the existence of $x_*$ as in Theorem~\ref{thm: boutroux condition}, not on its uniqueness.
			\item In particular, for any other $x_*$ as in Theorem~\ref{thm: boutroux condition}, say $\widehat x_*$, we will be able to construct an associated measure $\widehat \mu_*$ as in Theorem~\ref{thm: density zero distribution} and verify the convergence \eqref{eq: weak convergence zeros} with $\widehat \mu_*$ instead of $\mu_*$
			\item By uniqueness of the limiting zero distribution, we would have $\widehat \mu_*=\mu_*$, and as such their supports would have to agree. Hence, the endpoints of the supports agree as well, thus $x_*=\widehat x_*$.
		\end{enumerate}
		
		Finally, to show that $x_*\to 1$ as $\lambda \to \infty$, we start with a change of variables in \eqref{eq: psi lambda defn} to arrive at
		\begin{equation*}
		\psi\left( x\right) = \re \int_{-1}^0 \left(1-x-\frac{2i}{\lambda}\right) Q^{1/2}_+\left( \left(1-x-\frac{2i}{\lambda}\right)s+1\right)\,ds,
		\end{equation*}
		where we are integrating over the branch cut oriented on the real axis from $-1$ to $0$. Another cumbersome calculation shows that
		\begin{equation*}
		\left(1-x-\frac{2i}{\lambda}\right) Q^{1/2}_+\left( \left(1-x-\frac{2i}{\lambda}\right)s+1\right) = \lambda c_1 + c_0 + \mathcal{O}\left(\frac{1}{\lambda}\right), \qquad \lambda \to \infty, 
		\end{equation*}
		where
		\begin{equation*}
		c_1 = \frac{i\left(x-1\right)}{2}\left(\frac{\left(s+1\right)\left(-1-s-x+sx\right)}{s\left(-2-s+sx\right)}\right)^{1/2}_+,
		\end{equation*}
		and 
		\begin{equation*}
		c_0=\frac{3+4s+s^2+x-4sx-2s^2x+s^2x^2}{\left(-2-s+sx\right)\left(-1-s-x+sx\right)} \left(\frac{\left(s+1\right)\left(-1-s-x+sx\right)}{s\left(-2-s+sx\right)}\right)^{1/2}_+,
		\end{equation*}
		with a uniform error term for $x$ in compact subsets of $\R$.
		
		If now, in particular, we restrict to $x\in \left(0,1\right)$, we see that
		\begin{equation*}
		\left(\frac{\left(s+1\right)\left(-1-s-x+sx\right)}{s\left(-2-s+sx\right)}\right)^{1/2}_+ \in i \mathbb{R}
		\end{equation*}
		for $s \in \left(-1,0\right)$, so in this case
		\begin{equation*}
		\psi\left(x\right) = \frac{i\lambda\left(x-1\right)}{2} \int_{-1}^0 \left(\frac{\left(s+1\right)\left(-1-s-x+sx\right)}{s\left(-2-s+sx\right)}\right)^{1/2}_+\,ds + \mathcal{O}\left(\frac{1}{\lambda}\right), \qquad \lambda \to \infty,
		\end{equation*}
		with uniform error for $x\in (0,1)$.
		This means that the coefficient with $\lambda$ in the right-hand side above is nonzero if and only if $x\neq 1$. However, we must have $x_*\in (0,1)$ and $\psi(x_*)=0$ for every $\lambda>\lambda_c$, which, by virtue of the expansion above, can only happen if $x_*\to 1$ as $\lambda\to \infty$, as desired.
	\end{proof}
	
	Actually, we showed that
	\begin{equation}\label{eq: order xstar}
	x_* = 1 + o\left(\frac{1}{\lambda}\right), \qquad \lambda \to \infty,
	\end{equation}
	because the right-hand side of the last identity above, when evaluated at $x_*$, must be $0$ also in the limit $\lambda\to \infty$.
	
\begin{remark}	\label{rmk:uniqueness_x}
Along the same lines as we indicated on how to proceed for the proof of uniqueness above, we could in fact verify that $x_*\to 0$ as $\lambda\to \lambda_c^+$ as follows.

 First, one should observe that for $\lambda=\lambda_c$ the equation \eqref{eq: boutroux condition} is solved by $x_*=0$; this is a consequence of the genus zero study by Deaño \cite{deano2014large}. Following very closely the asymptotic analysis presented therein (with the only difference that a local parametrix at $z_*\in i\R$ would have to be considered as well), it does follow that an asymptotic formula as \eqref{eq: asymptotics outside gamma} takes place, say along the whole subsequence of polynomials $p_{2n}^{\lambda}$ of even degree. Consequently, this would imply the zero convergence of the whole sequence of counting measures for $p_{2n}^{\lambda}$ towards the measure determined by the conditions \eqref{eq: boutroux condition} and \eqref{eq: pretrajectory1}--\eqref{eq: pretrajectory2} with $x_*=0$, so in particular in this case the limiting support of zeros $\gamma_1\cup\gamma_2$ consists of one single arc, say $\gamma(x_*=0)$, that intersects the imaginary axis at $z_*$.

Back to the limit of $x_*(\lambda)$ as $\lambda \searrow \lambda_c$, observe that because $0\leq x_*\leq 1$ for any $\lambda>\lambda_c$ we can talk about accumulation points of $x_*=x_*(\lambda)$ as $\lambda\to\lambda_c$. Any such accumulation point, say $\widehat x_*$, would necessarily solve \eqref{eq: boutroux condition} with $\lambda=\lambda_c$, as this is a continuous condition on $x_*$. As such, for any such accumulation point we could perform the whole asymptotic analysis carried out in this paper, concluding at the end of the day the asymptotic formula \eqref{eq: asymptotics outside gamma} for the whole sequence $p_{2n}^\lambda$ with $\lambda=\lambda_c$ and this value $\widehat x_*$, obtaining also that the limiting support of zeros is $\gamma_1\cup\gamma_2=\gamma(\widehat x_*)$ obtained from $\widehat x_*$. 

Comparing the analysis just explained, we obtain that $\gamma(\widehat x_*)=\gamma(x_*=0)$, in particular $\gamma(\widehat x_*)$ is a single analytic arc, and then $\widehat x_*=0$ as wanted.

Despite the natural details needed to turn the sketch above into a rigorous statement, we should stress that the asymptotic analysis for $\lambda=\lambda_c$ is not yet rigorously completed in the literature. Such a proof would require, as mentioned above, an appropriate analysis of a local parametrix near the intersection point $z_*$ of $\gamma_1\cup\gamma_2$ that would occur in that case.  
\end{remark}	
	
	 To conclude this section, it is now useful to compute the first integral in \eqref{eq: periods}.
	
	\begin{prop} For $x_*$ given by Theorem~\ref{thm: boutroux condition}, we have that
		\begin{equation}\label{eq: half mass}
		\int_{z_*}^{1}Q^{1/2}_+(s)\; ds = \frac{\pi i}{2}.
		\end{equation}
	\end{prop}
	\begin{proof}
		Using \eqref{eq: boutroux condition} and Proposition~\ref{prop: symmetry integrals},
		$$
		2\int_{z_*}^1Q^{1/2}_+(s)\; ds=2i\im\int_{z_*}^1Q^{1/2}_+(s)\; ds=\int_{z_*}^1Q^{1/2}_+(s)\; ds+\int_{-1}^{-\overline z_*}Q^{1/2}_+(s)\; ds.
		$$
		Using deformation of contours on the right-hand side above, we get that
		$$
		2\int_{z_*}^1Q^{1/2}_+(s)\; ds=\frac{1}{2}\int_{\alpha}Q^{1/2}(s)\; ds,
		$$
		where $\alpha$ is a closed contour that encircles the whole branch cut $L$ in the clockwise direction. To compute the latter integral, we deform $\alpha$ to $\infty$ and use the expansion \eqref{eq: branch square root} to get
		$$
		\int_{z_*}^1Q^{1/2}_+(s)\; ds=\frac{1}{4}\;  2\pi i\res(Q^{1/2}(s),s=\infty)=\frac{\pi i}{2},
		$$
		as wanted.
	\end{proof}

\section{The Associated Quadratic Differential and its Trajectories}\label{sec: trajectories}

	The final goal of this section is to prove Theorem~\ref{thm: support zero distribution}. To construct the arcs $\gamma_1$ and $\gamma_2$, our main tool will be the theory of trajectories of quadratic differentials. We will keep the discussion of the basic theory here to a minimum, and refer to \cite[Appendix~B]{martinez_finkelshtein_silva2016} for details. The general theory can be found in the books by Strebel \cite{strebel_book} or Pommerenke \cite[Chapter~8]{pommerenke_book}.
	
	We call an arc $\tau\subset \C$ an {\it arc of trajectory} of the quadratic differential $-Qdz^2$ if for a fixed point $p\in \tau$,
	$$
	\int_p^z Q^{1/2}(s)ds\in i\mathbb R,\quad z\in \tau.
	$$
	Note that although the integral above depends on the starting point $p$ and on the branch of the square root taken, the condition that it is purely imaginary is independent of these. Also, we remark that this condition matches with \eqref{eq: pretrajectory1}--\eqref{eq: pretrajectory2}.
	
	Similarly, an arc $\tau$ is called an {\it arc of orthogonal trajectory} if for a base point $p\in \tau$,
	$$
	\int_p^z Q^{1/2}(s)ds\in \mathbb R,\quad z\in \tau.
	$$
	Maximal arcs of (orthogonal) trajectories are simply called {\it (orthogonal) trajectories}. 
	
	If $Q(p)\neq 0,\infty$, then there exists exactly one arc of trajectory and one arc of orthogonal trajectory passing through $p$, and these arcs are orthogonal to each other at $p$. In particular, different critical trajectories can only intersect at critical points. 
	
	Trajectories emanating from zeros or simple poles of $-Qdz^2$ are called critical. From each simple zero $p=z_*,-\overline z_*$ of $Q$ emanates exactly three critical trajectories, and the angle between consecutive ones is $2\pi/3$. From each simple pole $p=\pm 1$ emanates one critical trajectory.
	
	Finally, from the expansion 
	$$
	Q(z)=-\frac{\lambda^2}{4}+\Boh(z^{-1}),\quad z\to \infty,
	$$
	we see that the point $z=\infty$ is a pole of order $4$ of $-Qdz^2$, and also that any trajectory extending to $\infty$ has to do so along angles $0$ or $\pi$, that is, horizontally. Furthermore, by the general theory there has to be at least two critical trajectories ending at $\infty$, one along each allowable direction.
	
	We will need two basic principles that follow from the general theory of trajectories of quadratic differentials. These principles are thoroughly discussed in \cite[Section~4.5.1]{martinez_finkelshtein_silva2016}.
	
	\begin{itemize}
		\item[\bf P1.] If a critical trajectory $\tau$ emerges from a zero contained in a simply connected domain $D\subset \C$ that does not contain poles of $Qdz^2$, then either $\tau$ connects to another zero inside $D$, or $\tau$ intersects $\partial D$. In a similar spirit, if $D\subset \C$ is simply connected and contains exactly one pole, then the trajectory emanating from this pole has to either end at a zero inside $D$ or hit the boundary of $D$.
		\item[\bf P2.] If $D\subset \C$ is a simply connected domain whose boundary is a union of critical trajectories and it does not contain poles in its interior, then it has to contain at least one pole on its boundary.
	\end{itemize}
	
	The next principles are specific to our particular situation, and follow immediately from the explicit form of $Q$ in \eqref{def: rational Q}.
	
	\begin{itemize}
		\item[\bf P3.] Because $Qdz^2$ has three distinct poles, any critical trajectory has to connect two critical points (possibly the same).
		\item[\bf P4.] If $\tau$ is an arc of trajectory, then its reflection $-\tau^*$ onto the imaginary axis is also an arc of trajectory.
	\end{itemize}
	
	For a zero $p$ denote its order by $\eta(p)$, and for a pole $p$ of order $m>0$ set $\eta(p)=-m$. Fix a simply connected domain $D\subset \overline \C$ whose boundary is a finite union of critical trajectories. Given a critical point $p\in \partial D$, we set
	$$
	\beta(p)=1-\theta(p)\frac{\eta(p)+2}{2\pi},
	$$
	where $\theta(p)\in [0,2\pi]$ is the inner angle of $\partial D$ at $p$. For instance, if $p$ is a zero on $\partial D$, the value $\beta(p)$ gives the number of trajectories emanating from $p$ along the interior of $D$. 
	
	The following formula, valid for any simply connected domain $D$ as above,  is known as the Teichm\"uller Lemma \cite[Theorem~14.1]{strebel_book},
	\begin{equation}\label{teichmuller_formula}
	\sum_{p\in \partial D} \beta(p) = 2+\sum_{p\in D} \eta(p).
	\end{equation}
	
	The global behavior of all critical trajectories of $Qdz^2$ is described by the following Theorem.
	
	\begin{theorem}\label{theorem: critical graph}
		The trajectories of $-Qdz^2$ are symmetric with respect to reflection over the imaginary axis. Furthermore, the critical trajectories in the right half plane are as follows.
		\begin{enumerate}[(i)]
			\item There is a critical trajectory connecting $1$ and $z_*$.
			\item There is a critical trajectory connecting $z_*$ and $-\overline z_*$. This is the only trajectory from $z_*$ that moves to the left-half plane.
			\item The remaining trajectory that emanates from $z_*$ ends at $\infty$. Furthermore, $\arg z\to 0$ as $z\to \infty$ along this trajectory. 
		\end{enumerate}
	\end{theorem}
	
	These trajectories are numerically computed in Figure~\ref{fig: trajectories} for various choices of $\lambda$.
	
	\begin{figure}[t]
		\centering
		\begin{subfigure}[b]{0.32\linewidth}
			\includegraphics[width=\linewidth]{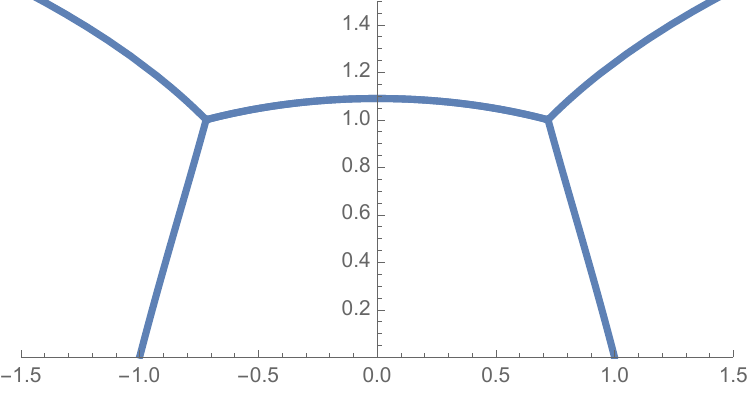}
			\caption*{$\lambda = 2$}
			\label{fig: subfig 1 trajectories}
		\end{subfigure}
		\begin{subfigure}[b]{0.32\linewidth}
			\includegraphics[width=\linewidth]{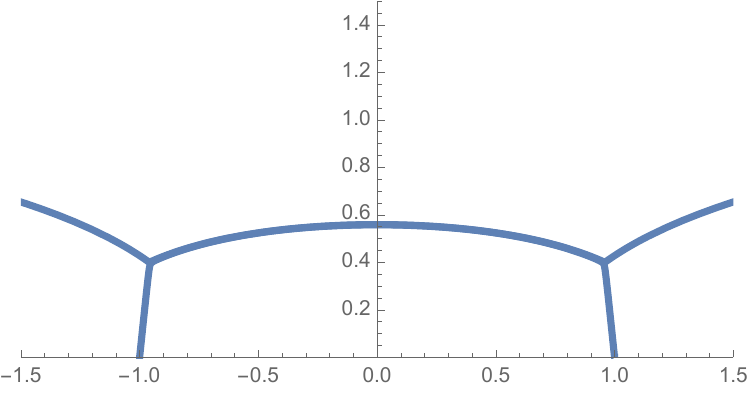}
			\caption*{$\lambda = 5$}
			\label{fig: subfig 2 trajectories}
		\end{subfigure}
		\begin{subfigure}[b]{0.32\linewidth}
			\includegraphics[width=\linewidth]{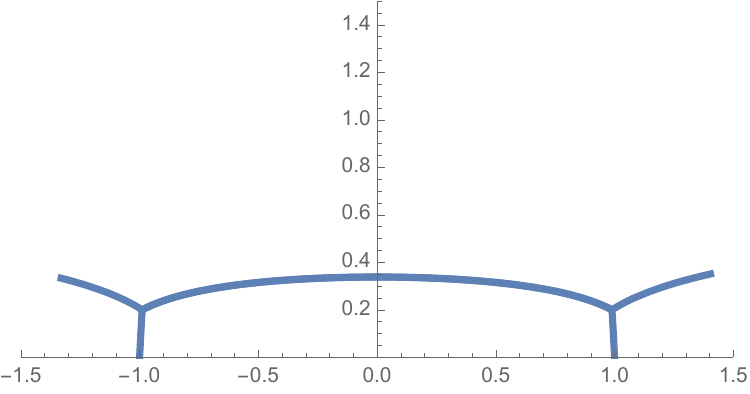}
			\caption*{$\lambda = 10$}
			\label{fig: subfig 3 trajectories}
		\end{subfigure}
		\caption{Trajectories of $-Q_\lambda(z)\, dz^2$.}
		\label{fig: trajectories}
	\end{figure}
	
	We split the proof of Theorem~\ref{theorem: critical graph} into several lemmas.
	
	\begin{lemma}\label{lem: trajectories 1}
		There is at least one trajectory emanating from $z_*$ with endpoint on $\{1,-\overline z_*\}$.
	\end{lemma}
	\begin{proof}
		To get to a contradiction, suppose that there is no trajectory as described. If a trajectory emanating from $z_*$ hits $i\R$, then using Principle P4 we conclude that this trajectory connects $z_*$ to $-\overline z_*$. But the latter cannot occur, so all three trajectories emanating from $z_*$ have to stay in the right half plane and none of them can end at $z=1$.
		
		Also, none of these trajectories from $z_*$ could be a closed loop on $\C$. Indeed, if this were the case, then Principle P2 applied to the bounded domain $D$ determined by this loop would guarantee that $z=1$ is inside the loop. But then the trajectory $\tau$ emerging from $z=1$, by Principle P1, would have to hit $\partial D$. As $\partial D$ is a trajectory and trajectories can only intersect at critical points, this means that $\tau$ would have to connect to the only critical point $z_*\in \partial D$, which we are assuming cannot occur.
		
		So this discussion and Principle P3 yield that all the trajectories from $z_*$ have to extend to $\infty$, and because they have to stay on the upper half plane they all have to do so along angle $0$. These three trajectories determine exactly two domains on the right half plane, whose angle at $\infty$ is exactly $0$. Certainly at least one of these domains is pole-free, and for this one domain
		$$
		\sum \eta(p) \geq 0 \quad \mbox{and} \quad \sum \beta(p)=1,
		$$
		in contradiction to \eqref{teichmuller_formula}. The proof is complete.
	\end{proof}
	
	\begin{lemma}\label{lem: trajectories 2}
		There cannot be two trajectories emanating from $z_*$ and intersecting $i\R$.
	\end{lemma}
	\begin{proof}
		Suppose that there are two such trajectories, say $\tau_1$ and $\tau_2$. By Principle P4, then they both have to connect to $-\overline z_*$, so their union is the boundary of a simply connected domain $D$. By Principles P2 and P4, the poles $z=\pm 1$ have to be in $D$.
		
		Consider now the remaining trajectory $\tau_3$ from $z_*$. If $\tau_3$ emerges inside $D$, then because we already know that $z=\pm 1\in D$ we would not have any critical trajectory extending to $z=\infty$, but this cannot occur. Hence, we have to have that $\tau_3$ extends to $\infty$. Using again Principle P3, we see that the trajectories from $z=z_*,-\overline z_*$ are fully determined as in Figure~\ref{fig: lemma trajectories 2 a}.
		
		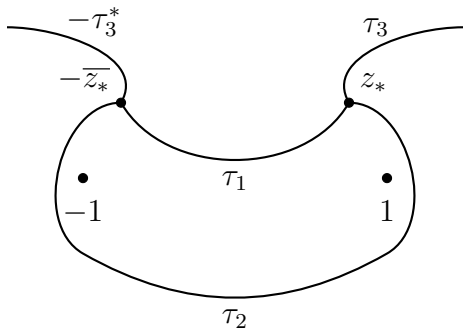
\begin{figure}[t]
			\centering
			\begin{tikzpicture}[scale=1]
			\draw [thick] (-1.5,1) to [out=60,in=0,edge node={node [pos=0.5,above] {$-\tau_3^*$}}] (-3,2);
			\draw [thick] (1.5,1) to [out=180-60,in=180,edge node={node [pos=0.5,above] {$\tau_3$}}] (3,2);
			\draw [thick] (-1.5,1) to [out=-60,in=180+60,edge node={node [pos=0.5,below] {$\tau_1$}}] (1.5,1);
			\draw [thick] (-1.5,1) to [out=180,in=150] (-2,-1)
			to [out=180+150,in=-150,edge node={node [pos=0.5,below] {$\tau_2$}}] (2,-1)
			to [out=-150+180,in=0] (1.5,1);		 						
			\draw [fill] (1.5,1) circle [radius=0.06];
			\draw [fill] (-1.5,1) circle [radius=0.06];
			\draw [fill] (2,0) circle [radius=0.06];
			\draw [fill] (-2,0) circle [radius=0.06];
			\node [below] at (-2,-0.15) {$-1$};
			\node [below] at (2,-0.15) {$1$};
			\node [above left] at (-1.5,1) {$-\overline{z_*}$};
			\node [above right] at (1.5,1) {$z_*$};
			\end{tikzpicture}
			\caption{The trajectories of $-Qdz^2$ as used in the proof of Lemma~\ref{lem: trajectories 2}.}\label{fig: lemma trajectories 2 a}
		\end{figure}
		
		We now look at the trajectory $\tau$ emanating from $z=1$. By Principle P3, we conclude that $\tau$ has to connect $z=1$ and $z=-1$, so the full critical graph of $-Qdz^2$ is now depicted in Figure~\ref{fig: lemma trajectories 2 b}.
		
		\begin{figure}[t]
			\centering
			\begin{tikzpicture}[scale=1]
			\draw [thick] (-1.5,1) to [out=60,in=0,edge node={node [pos=0.5,above] {$-\tau_3^*$}}] (-3,2);
			\draw [thick] (1.5,1) to [out=180-60,in=180,edge node={node [pos=0.5,above] {$\tau_3$}}] (3,2);
			\draw [thick] (-1.5,1) to [out=-60,in=180+60,edge node={node [pos=0.5,above] {$\tau_1$}}] (1.5,1);
			\draw [thick] (-1.5,1) to [out=180,in=150] (-2,-1)
			to [out=180+150,in=-150,edge node={node [pos=0.5,below] {$\tau_2$}}] (2,-1)
			to [out=-150+180,in=0] (1.5,1);		 						
			\draw [thick] (-2,0) to [bend right=20,edge node={node [pos=0.5,above] {$\tau$}}] (2,0);
			\draw [fill] (1.5,1) circle [radius=0.06];
			\draw [fill] (-1.5,1) circle [radius=0.06];
			\draw [fill] (2,0) circle [radius=0.06];
			\draw [fill] (-2,0) circle [radius=0.06];
			\node [below] at (-2,-0.15) {$-1$};
			\node [below] at (2,-0.15) {$1$};
			\node [above left] at (-1.5,1) {$-\overline{z_*}$};
			\node [above right] at (1.5,1) {$z_*$};
			\node at (0,-.9) {$\mathcal D$};
			\end{tikzpicture}
			\caption{The (hypothetical) critical graph of $-Qdz^2$ as used in the proof of Lemma~\ref{lem: trajectories 2}.}\label{fig: lemma trajectories 2 b}
		\end{figure}
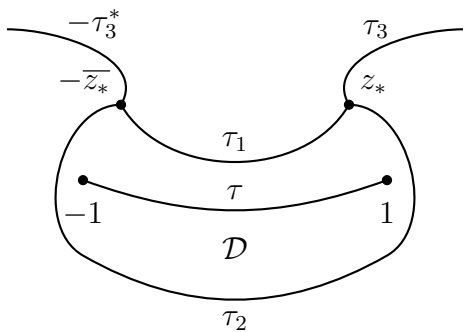 
		
		Consider the domain $\mathcal D$, which is obtained by removing $\tau$ from the domain bounded by $\tau_1\cup\tau_2$, see Figure~\ref{fig: lemma trajectories 2 b}. According to the canonical decomposition of the critical graph \cite[Theorem~B1]{martinez_finkelshtein_silva2016}, the domain $\mathcal D$ is a ring domain, which means that for the function
		$$
		\Upsilon(z)=\int_{1}^z \sqrt{Q(s)}ds
		$$
		and some nonzero real constant $c$, the map $F(z)=e^{c\Upsilon(z)}$ is a conformal map from $\mathcal D$ to an annulus of positive radii $r<R$, and with boundary correspondence $F(\tau)=\partial D_r(0)$ and $F(\tau_2\cup\tau_1)=\partial D_R(0)$. In particular, clearly $F(1)=1$ so $r=|F(1)|=1$ and hence $R=|F(z_*)|>1$. On the other hand, using \eqref{eq: boutroux condition} we get also that $|F(z_*)|=1$, a contradiction, thereby concluding the proof.
	\end{proof}

	\begin{lemma}\label{lem: trajectories 3}
		If there is a trajectory connecting $1$ and $z_*$, then there is a trajectory connecting $z_*$ and $-\overline z_*$.
	\end{lemma}
	\begin{proof}
		Let $\tau_1$ be the trajectory connecting $z_*$ and $z=1$, and $\tau_2$ and $\tau_3$ the remaining two trajectories emanating from $z_*$. The proof will again proceed by contradiction. If we assume that there is no trajectory connecting $z_*$ to $-z_*$, then $\tau_1$ and $\tau_2$ have to extend to $z=\infty$ horizontally on the right-half plane. Using the symmetry on Principle P4, there are only two possibilities left for the critical graph of $Qdz^2$, depending on whether $z=1$ belongs or not to the domain bounded by $\tau_2$ and $\tau_3$ on the right half plane. These two possibilities are displayed in Figure~\ref{fig: lemma trajectories 3 a}.
		
		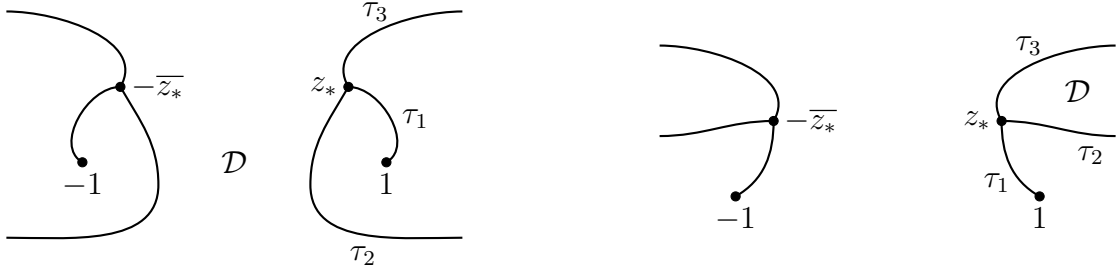
\begin{figure}[t]
			\begin{subfigure}[c]{0.5\textwidth}
				\centering
				\begin{tikzpicture}[scale=1]
				%
				\draw [thick] (-1.5,1) to [out=60,in=0] (-3,2);
				\draw [thick] (1.5,1) to [out=180-60,in=180,edge node={node [pos=0.5,above] {$\tau_3$}}] (3,2);
				\draw [thick] (-1.5,1) to [out=-60, in=90] (-1,-0.3)
				to [out=90+180,in=0] (-3,-1);
				\draw [thick] (1.5,1) to [out=180+60, in=180-90] (1,-0.3)
				to [out=-90,in=180,edge node={node [pos=0.5,below] {$\tau_2$}}] (3,-1);
				\draw [thick] (-1.5,1) to [out=180,in=150] (-2,0);	
				\draw [thick] (1.5,1) to [out=180-180,in=180-150,edge node={node [pos=0.5,right] {$\tau_1$}}] (2,0);	
				\draw [fill] (1.5,1) circle [radius=0.06];
				\draw [fill] (-1.5,1) circle [radius=0.06];
				\draw [fill] (2,0) circle [radius=0.06];
				\draw [fill] (-2,0) circle [radius=0.06];
				\node [below] at (-2,0) {$-1$};
				\node [below] at (2,0) {$1$};
				\node [right] at (-1.5,1) {$-\overline{z_*}$};
				\node [left] at (1.5,1) {$z_*$};
				\node at (0,0) {$\mathcal D$};
				\end{tikzpicture}
			\end{subfigure}%
			\begin{subfigure}[c]{0.5\textwidth}
				\centering
				\begin{tikzpicture}[scale=1]
				%
				\draw [thick] (-1.5,1) to [out=60,in=0] (-3,2);
				\draw [thick] (1.5,1) to [out=180-60,in=180,edge node={node [pos=0.5,above] {$\tau_3$}}] (3,2);
				\draw [thick] (-1.5,1) to [out=180+90, in=180-150] (-2,0);
				\draw [thick] (1.5,1) to [out=-90,in=150,edge node={node [pos=0.75,left] {$\tau_1$}}] (2,0);
				\draw [thick] (-1.5,1) to [out=180,in=0] (-3,0.8);	
				\draw [thick] (1.5,1) to [out=180-180,in=180,edge node={node [pos=0.8,below] {$\tau_2$}}] (3,0.8);	
				\draw [fill] (1.5,1) circle [radius=0.06];
				\draw [fill] (-1.5,1) circle [radius=0.06];
				\draw [fill] (2,0) circle [radius=0.06];
				\draw [fill] (-2,0) circle [radius=0.06];
				\node [below] at (-2,0) {$-1$};
				\node [below] at (2,0) {$1$};
				\node [right] at (-1.5,1) {$-\overline{z_*}$};
				\node [left] at (1.5,1) {$z_*$};
				\node at (2.5,1.4) {$\mathcal D$};
				\end{tikzpicture}
			\end{subfigure}
			\caption{The hypothetical possibilities for the critical graph of $-Qdz^2$, as derived in the proof of Lemma~\ref{lem: trajectories 3}.}\label{fig: lemma trajectories 3 a}
		\end{figure} 
		
		In either of the two situations, we consider the domain $\mathcal D$ on the critical graph uniquely determined by the condition that $\partial \mathcal D\cap \tau_1=\{z_*\}$. This domain is also represented in Figure~\ref{lem: trajectories 3}. Applying \eqref{teichmuller_formula} to $\mathcal D$, we see that in the situation represented in Figure~\ref{fig: lemma trajectories 3 a}-left, we have
		$$
		\sum_{p\in \partial D} \beta(p)=4,\quad \sum_{p\in \mathcal D}\eta(p)=0,
		$$
		whereas on the situation depicted in Figure~\ref{fig: lemma trajectories 3 a}-right,
		$$
		\sum_{p\in \partial D} \beta(p)=1,\quad \sum_{p\in \mathcal D}\eta(p)=0.
		$$
		Hence, neither of these situations can happen, concluding the contradiction argument and the proof.
	\end{proof}
	
	\begin{lemma}\label{lem: trajectories 4}
		There is a trajectory connecting $1$ and $z_*$.
	\end{lemma}
	\begin{proof}
		Again to get to a contradiction, suppose there is no such trajectory. By Lemma~\ref{lem: trajectories 1} there has to be a trajectory $\tau_1$ connecting $z_*$ and $-\overline z_*$, and then by Lemma~\ref{lem: trajectories 2} and Principle P3, the other two trajectories $\tau_2$ and $\tau_3$ stay on the right half plane and have to extend to $\infty$ with angle $0$. 
		
		These trajectories $\tau_2$ and $\tau_3$ determine a simply connected domain $\mathcal H$ on the right half plane, and this domain does not contain $\tau_1$. Hence, applying \eqref{teichmuller_formula} to this domain, we arrive at
		$$
		\sum_{p\in \partial D}\beta(p)=\beta(z_*)+\beta(\infty)=1,
		$$
		and thus $\mathcal H$ must contain the only pole $z=1$ on the right half plane. This means that the critical graph is as depicted in Figure~\ref{fig: lemma trajectories 4 a}.
		
		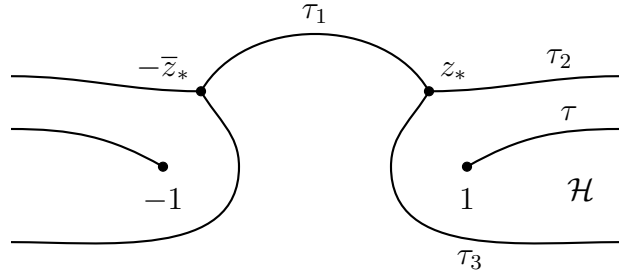
\begin{figure}[t]
			\centering
			\begin{tikzpicture}[scale=1]
			\draw [thick] (-1.5,1) to [out=60,in=180-60,edge node={node [pos=0.5,above] {$\tau_1$}}] (1.5,1);
			\draw [thick] (1.5,1) to [out=180+60,in=180-90] (1,0)
			to [out=180-270,in=180,edge node={node [pos=0.5,below] {$	\tau_3$}}] (4,-1);
			\draw [thick] (-1.5,1) to [out=-60,in=90] (-1,0)
			to [out=270,in=0] (-4,-1);
			\draw [thick] (1.5,1) to [out=0,in=180,edge node={node [pos=0.7,above] {$\tau_2$}}] (4,1.2);		 				
			\draw [thick] (-1.5,1) to [out=180,in=0] (-4,1.2);		
			\draw[thick] (2,0) to [out=30,in=180,edge node={node [pos=0.7,above] {$\tau$}}] (4,0.5);
			\draw[thick] (-2,0) to [out=180-30,in=0] (-4,0.5);
			\draw [fill] (1.5,1) circle [radius=0.06];
			\draw [fill] (-1.5,1) circle [radius=0.06];
			\draw [fill] (2,0) circle [radius=0.06];
			\draw [fill] (-2,0) circle [radius=0.06];
			\node [below] at (-2,-0.15) {$-1$};
			\node [below] at (2,-0.15) {$1$};
			\node [above left] at (-1.5,1) {$-\overline{z}_*$};
			\node [above right] at (1.5,1) {$z_*$};
			\node at (3.5,-0.3) {$\mathcal H$};
			\end{tikzpicture}
			\caption{The trajectories of $-Qdz^2$ as used in the proof of Lemma~\ref{lem: trajectories 4}.}\label{fig: lemma trajectories 4 a}
		\end{figure}
		
		We continue focusing on the domain $\mathcal H$. According to the canonical decomposition of the critical graph \cite[Theorem~B1]{martinez_finkelshtein_silva2016}, this domain $\mathcal H$ is a strip domain, which means that for some branch of the square root, the function
		$$
		\Upsilon(z)=\int_{z_*}^z\sqrt{Q(s)}ds
		$$
		is a conformal map from $\mathcal H$ to a strip of the form
		$$
		\mathcal S=\{z \in \C \; \mid \; c_1< \re z <c_2\},
		$$
		and $\Upsilon$ extends continuously to the boundary of $\mathcal H$, hence mapping $\tau_2\cup\tau_3$ and $\tau$ to distinct connected components of the strip $\mathcal S$.
		
		Clearly, $\Upsilon(z_*)=0$ and because of the boundary correspondence just explained we must then have $\re \Upsilon(1)\neq 0$. However, using \eqref{eq: boutroux condition} we actually see that $\re\Upsilon(1)=0$, a contradiction. The proof is complete.
	\end{proof}
	
	\begin{proof}[Proof of Theorem~\ref{theorem: critical graph}]
		The symmetry under reflection is nothing but Principle P4.
		
		A combination of Lemmas~\ref{lem: trajectories 3} and \ref{lem: trajectories 4} immediately yield (i) and (ii). From Lemma~\ref{lem: trajectories 2} we see that the remaining trajectory emanating from $z_*$, say $\tau_3$, has to stay on the right half plane, so it has to extend to $\infty$ with angle $0$. This concludes the proof of (iii).
	\end{proof}
	
	In principle, the trajectory in part (iii) of Theorem~\ref{theorem: critical graph} could extend to $\infty$ going below $z=-1$, which would result on the critical graph displayed in Figure~\ref{fig: trajectories almost final}, left (that we call case (i)), in contrast with Figure~\ref{fig: trajectories almost final}, right (that we call case (ii)). We now justify why case (ii) instead of (i) takes place, which corresponds exactly with the numerical outputs in Figure~\ref{fig: trajectories}.
	
	\begin{figure}[t]
		\begin{subfigure}[c]{0.5\textwidth}
			\centering
			\begin{tikzpicture}[scale=.8]
			\draw [thick] (-1.5,1) to [out=60,in=180-60,edge node={node [pos=0.5,above] {$\tau_1$}}] (1.5,1);
			\draw [thick] (1.5,1) to [out=180+60,in=180-90] (1,0)
			to [out=180-270,in=180,edge node={node [pos=0.5,below] {$	\tau_3$}}] (4,-1);
			\draw [thick] (-1.5,1) to [out=-60,in=90] (-1,0)
			to [out=270,in=0] (-4,-1);
			\draw [thick] (1.5,1) to [out=0,in=100,edge node={node [pos=0.5,right] {$\tau_2$}}] (2,0);		 				
			\draw [thick] (-1.5,1) to [out=180,in=80] (-2,0);		
			\draw [fill] (1.5,1) circle [radius=0.06];
			\draw [fill] (-1.5,1) circle [radius=0.06];
			\draw [fill] (2,0) circle [radius=0.06];
			\draw [fill] (-2,0) circle [radius=0.06];
			\node [below] at (-2,-0.15) {$-1$};
			\node [below] at (2,-0.15) {$1$};
			\node [above left] at (-1.5,1) {$-\overline{z}_*$};
			\node [above right] at (1.5,1) {$z_*$};
			\node at (3.5,-0.3) {$\mathcal H$};
			\end{tikzpicture}
		\end{subfigure}%
		\begin{subfigure}[c]{0.5\textwidth}
			\centering
			\begin{tikzpicture}[scale=.8]
			\draw [thick] (-1.5,1) to [out=10,in=170,edge node={node [pos=0.5,above] {$\tau_1$}}] (1.5,1);
			\draw [thick] (1.5,1) to [out=50,in=180,edge node={node [pos=0.5,above] {$\tau_3$}}] (4,1.7);
			\draw [thick] (-1.5,1) to [out=180-50,in=0] (-4,1.7);								   
			\draw [thick] (1.5,1) to [out=290,in=100,edge node={node [pos=0.5,left] {$\tau_2$}}] (2,-0.5);		 				
			\draw [thick] (-1.5,1) to [out=180-290,in=180-100] (-2,-0.5);		
			\draw [fill] (1.5,1) circle [radius=0.06];
			\draw [fill] (-1.5,1) circle [radius=0.06];
			\draw [fill] (2,-0.5) circle [radius=0.06];
			\draw [fill] (-2,-0.5) circle [radius=0.06];
			\node [below] at (-2,-0.5) {$-1$};
			\node [below] at (2,-0.5) {$1$};
			\node [left] at (-1.5,1) {$-\overline{z}_*$};
			\node [right] at (1.5,1) {$z_*$};
			\node at (3,-0.3) {$\mathcal H$};
			\end{tikzpicture}
		\end{subfigure}%
		\caption{The two possible configurations for the critical graph of $Qdz^2$ after the proof of Theorem~\ref{theorem: critical graph}. It turns out that the only correct configuration is actually the one on the right, as shown in the comments after the proof. We invite the reader to compare with the numerical output produced in Figure~\ref{fig: trajectories}. }\label{fig: trajectories almost final}
	\end{figure}
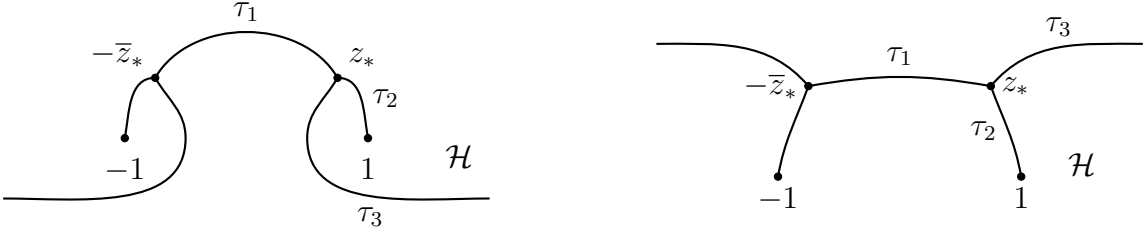

	In either case, let $Q^{1/2}$ be the branch of the square root defined by the asymptotics \eqref{eq: branch square root} and with branch cut being 
	$$
	\tau=\tau_1\cup\tau_2\cup \widehat \tau_2,
	$$
	where $\widehat \tau_2$ is the reflection of $\tau_2$ onto the imaginary axis, see Figure~\ref{fig: trajectories almost final}. We orient this branch cut from $-1$ to $1$. In either situation in Figure~\ref{fig: trajectories almost final}, consider the domain $\mathcal H$ bounded by $\tau\cup\tau_3\cup\widehat{\tau_3}$, with $\widehat \tau_3$ being the reflection of $\tau_3$ onto the imaginary axis. In other words, in case (i) $\mathcal H$ is the domain ``above'' the critical graph, and in case (ii) it is the domain ``below'' the critical graph, see Figure~\ref{fig: trajectories almost final}.
	
	From the general theory, we know that the function
	$$
	\Psi(z)=\int_{1}^zQ^{1/2}(s)\; ds ,\quad z\in \mathcal H,
	$$
	is a conformal map from $\mathcal H$ to either the left or the right half plane. By looking at the asymptotics \eqref{eq: branch square root}, we see that
	$$
	\Psi(z)=-\frac{i\lambda}{2}z\left(1+\boh(1)\right),
	$$
	so when $z\to \infty$ along $\mathcal H\cap i\R$ we must either have that $\re \Psi(z)\to +\infty$ (in case (i)) or $\re \Psi(z)\to -\infty$ (in case (ii)). This means that the image $\Psi(\mathcal H)$ is either the right or left half plane, corresponding to cases (i) or (ii), respectively.
	
	By boundary correspondence, we must then have $\Psi_+(\tau_2)\subset i\R_+$ in case (i) and $\Psi_+(\tau_2)\subset i\R_-$ in case (ii), which implies that $\im \Psi_+(z_*)> 0$ and $\im \Psi_+(z_*)< 0$, respectively. But from \eqref{eq: half mass} we know that $\Psi_+(z_*)=-\pi i/2$, so case (ii) has to be taking place.
	
	To conclude this section, we now prove Theorem~\ref{thm: support zero distribution} and the first part of Theorem~\ref{thm: density zero distribution}, concerning the construction of $\mu_*$.
	
	\begin{proof}[Proof of Theorem~\ref{thm: support zero distribution}]
		Set $\tau_2=\gamma_2$ as the trajectory as labeled in Figure~\ref{fig: trajectories almost final}, right, and $\gamma_1$ to be its reflection onto the imaginary axis. Then \eqref{eq: pretrajectory1}--\eqref{eq: pretrajectory2} hold by the definition of a critical trajectory. The uniqueness follows by the uniqueness of the critical graph of $Q^{1/2}dz$, and the fact that $\gamma_2$ is the only trajectory connecting $1$ and $z_*$.
	\end{proof}
	
	For the rest of the paper, we now use the branch cut structure along $\gamma=\gamma_1\cup\gamma_2$, as discussed right after the statement of Theorem~\ref{thm: support zero distribution}.
	
	\begin{proof}[Proof of Theorem~\ref{thm: density zero distribution}]
		By the definition of $\gamma_1$ and $\gamma_2$ as trajectories, we know that $\mu_*$ is a real measure. Also, because its density does not vanish on each of the arcs, we also know that $\mu_*$ cannot change sign in each of these arcs. From \eqref{eq: half mass},
		$$
		\mu_*(\gamma_2)=\frac{1}{\pi i}\int_{z_*}^1 Q^{1/2}_+(s)ds=\frac{1}{2},
		$$
		so $\mu_*$ has to be positive along $\gamma_2$. By symmetry $\mu_*(\gamma_1)=\mu_*(\gamma_2)$ (see for instance Proposition~\ref{prop: symmetry integrals}), so $\mu_*$ has to be a probability measure.
		
		To show that the shifted Cauchy transform satisfies the algebraic equation, we use that $Q^{1/2}_+=-Q^{1/2}_-$ along $\gamma_1\cup\gamma_2$ to write
		$$
		C^{\mu_*}(z)=\frac{1}{2 \pi i}\oint_C \frac{Q^{1/2}(s)}{s-z}ds,\quad z\in \C\setminus (\gamma_1\cup\gamma_2),
		$$
		where $C$ is a bounded contour that encircles $\gamma_1\cup\gamma_2$ in the clockwise direction and does not encircle $z$. Using \eqref{eq: branch square root} to compute residues,
		$$
		C^{\mu_*}(z) = \res_{s=z}\frac{Q^{1/2}(s)}{s-z} + \res_{s=\infty}\frac{Q^{1/2}(s)}{s-z}= Q^{1/2}(z) + \frac{i \lambda}{2},
		$$
		which is equivalent to \eqref{eq: spectral curve 1}. 
	\end{proof}

\section{Asymptotic Analysis - Part I}\label{sec: rhp}
	The goal of this section is to introduce the Fokas-Its-Kitaev Riemann-Hilbert formulation \cite{fokas1992isomonodromy} of the kissing polynomials, and to start the implementation of the Deift-Zhou nonlinear steepest descent analysis \cite{deift1999strong,deift1999uniform}, which will allow us to eventually provide uniform asymptotics of the kissing polynomials. 
	For convenience, we will use the notation
	$$
	E_{12}:=\begin{pmatrix}
	0 & 1 \\ 
	0 & 0
	\end{pmatrix}
	\quad
	\mbox{and}
	\quad 
	E_{21}:=\begin{pmatrix}
	0 & 0 \\
	1 & 0
	\end{pmatrix}.
	$$
	
	Define a contour $\Gamma_Y:= \gamma_1\cup\hat{\gamma}\cup\gamma_2$, oriented as in Figure~\ref{fig: Gamma Y}, where $\gamma_1$ and $\gamma_2$ are as in Theorem~\ref{thm: support zero distribution}. The contour $\widehat{\gamma}$ connects $-\overline{z_*}$ to $z_*$ and lies entirely within the sector determined by the trajectories of $Q\,dz^2$ which is above the trajectory $\tau_1$, see Figure~\ref{fig: Gamma Y}.
	\begin{figure}[t]
		\centering
		\begin{tikzpicture}[scale=1]
		\draw [dashed,thick] (-1.5,1) to [out=10,in=170,edge node={node [pos=0.5,below] {$\tau_1$}}] (1.5,1);
		\draw [dashed,thick] (1.5,1) to [out=50,in=180] (4,1.7);
		\draw [dashed,thick] (-1.5,1) to [out=180-50,in=0] (-4,1.7);								   
		\draw [ultra thick,postaction={mid arrow={black,scale=1.5}}] (1.5,1) to [out=290,in=100,edge node={node [pos=0.5,right] {$\gamma_2$}}] (2,-0.5);		 				
		\draw [ultra thick,postaction={rmid arrow={black,scale=1.5}}] (-1.5,1) to [out=180-290,in=180-100,edge node={node [pos=0.5,left] {$\gamma_1$}}] (-2,-0.5);	
		\draw [ultra thick,postaction={mid arrow={black,scale=1.5}}] (-1.5,1) to [bend left=60,edge node={node [pos=0.5,above] {$\widehat \gamma$}}] (1.5,1);
		\draw [fill] (1.5,1) circle [radius=0.06];
		\draw [fill] (-1.5,1) circle [radius=0.06];
		\draw [fill] (2,-0.5) circle [radius=0.06];
		\draw [fill] (-2,-0.5) circle [radius=0.06];
		\node [below] at (-2,-0.5) {$-1$};
		\node [below] at (2,-0.5) {$1$};
		\node [left] at (-1.5,1) {$-\overline{z}_*$};
		\node [right] at (1.5,1) {$z_*$};
		\end{tikzpicture}
		\caption{The contour $\Gamma_Y=\gamma_1\cup\widehat\gamma\cup\gamma_2$ in solid lines, and the trajectories of $-Q\,dz^2$ that are not in $\Gamma_Y$ in dashed lines, compare with Figure~\ref{fig: trajectories almost final}. }\label{fig: Gamma Y}
	\end{figure}
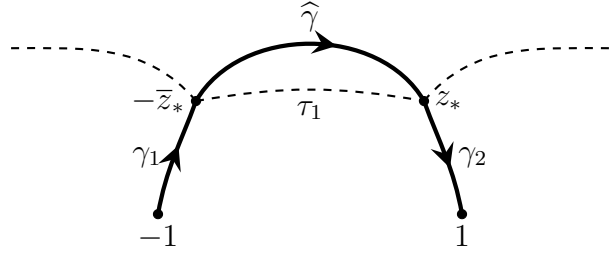	
	
	
	The Riemann-Hilbert problem for the kissing polynomials is formulated as follows. We seek a matrix function $Y:\C\setminus \Gamma_Y \to \mathbb{C}^{2\times 2}$ such that
	\begin{subequations}\label{eq: original RHP}
		\begin{alignat}{2}
		&Y(z) \text{ is analytic in } \mathbb{C}\setminus \Gamma_Y \\
		&Y_+(z) = Y_-(z)J_Y(z),\quad J_Y(z):= I +e^{-nV(z)}E_{12}, \qquad \qquad && z \in \Gamma_Y, \label{eq: Y jump}\\
		&Y(z) = \left(I + \mathcal{O}\left(\frac{1}{z}\right)\right) z^{n\sigma_3}, \qquad && z \to \infty.
		\end{alignat}
	\end{subequations}
	Here, $\sigma_3=\diag(1,-1)$ is the third Pauli matrix and $V(z)=-i\lambda z$ is as in \eqref{eq: varying weight kissing polynomials}.
	
	By the now standard theory, it follows that the solution to this problem, if it exists, is unique. The existence of $Y$ is equivalent to the existence of the kissing polynomial $p_n^\lambda$ and, furthermore, if $\kappa_{n-1}$ is finite and non-zero then $Y$ is explicitly given by
	\begin{equation}\label{eq: Y eqn}
	Y(z) = \begin{pmatrix}
	p_n^\lambda(z) & \left(\mathcal{C}p_n^\lambda e^{-nV}\right)(z) \\
	-2\pi i \kappa_{n-1}^2 p_{n-1}^\lambda(z) & -2\pi i \kappa_{n-1}^2 \left(\mathcal{C}p_{n-1}^\lambda e^{-nV}\right)(z)
	\end{pmatrix},
	\end{equation}
	where $\kappa_{n}$ is the normalizing constant for $p_n^\lambda$, obtained via
	$$
	\int_{-1}^{1}(p_n^\lambda(z))^2e^{-nV(z)}dz=\frac{1}{\kappa_n^2}.
	$$	
	In \eqref{eq: Y eqn}, $\mathcal{C}f$ denotes the Cauchy transform of the function $f$ along $\Gamma_Y$, i.e.
	\begin{equation*}
	\left(\mathcal{C}f\right)(z) = \frac{1}{2\pi i}\int_{\Gamma_Y} \frac{f(s)}{s-z}\, ds,
	\end{equation*}
	which is analytic in $\mathbb{C}\setminus \Gamma_Y$. 
	
	The Deift-Zhou steepest descent method consists of finding a sequence of explicit and invertible transformations
	$$
	Y\mapsto \cdots \mapsto R
	$$
	so that, at the end of the day, $R$ satisfies a new Riemann-Hilbert problem. This new Riemann-Hilbert problem should have jumps which decay to the identity as $n\to \infty$, and solutions to this RHP should also decay to the identity as $z\to\infty$.  In this case, $R$ can be obtained asymptotically as $n\to \infty$, with explicit control over error terms. By tracing back all the transformations, we can extract asymptotics for $Y$.
	
	There is now a vast literature on this asymptotic method, many of which are very similar to our present situation \cite{deift1999orthogonal,deift1999uniform,deano2014large}. In our case, the only nonstandard construction is in the so-called global parametrix, so we will only very briefly discuss the other steps, and provide more details in the construction of this global parametrix.

	\subsection{First Transformations}	
	From now on, we fix the branch of $Q^{1/2}(z)$ to be the one with cuts on $\gamma_1$ and $\gamma_2$, as explained after Theorem~\ref{thm: support zero distribution}, and define $\phi$ as in \eqref{eq: phi 4 eqn}, where the path of integration connects $1$ to $z$ without crossing $\left(-\infty, -1\right)\cup \gamma\cup\hat{\gamma}$. This function $\phi$ satisfies
	\begin{equation}\label{eq: properties phi 1}
	\begin{aligned}
	& \phi_+(z)+\phi_-(z)=(-1)^{j+1}i\kappa,& z\in \gamma_j, \\
	& \phi_+(z)-\phi_-(z)=-\pi i, & z\in \widehat \gamma, \\
	& \re \phi_+(z)>0, & z\in \widehat \gamma, 
	\end{aligned}
	\end{equation}
	and also the inequality
	\begin{equation}\label{eq: properties phi 2}
	\re\phi(z)<0 
	\end{equation}
is valid in the immediate vicinity of any subarc of $\gamma_1\cup\gamma_2$ that does not contain their endpoints.	
	
	Recall also the complex constant $l$ determined via the expansion \eqref{eq: expansion phi}. The first transformation, which aims to normalize $Y$ at $\infty$, reads
	\begin{equation}\label{eq: T eqn}
	T(z) = e^{n\left(l-\frac{i\kappa}{2}\right)\sigma_3 } \, Y(z) \, e^{n\left(\phi(z)-\frac{V(z)}{2}\right)\sigma_3}.
	\end{equation}
	With $\Gamma_T:=\Gamma_Y$, $T$ is analytic on $\C\setminus \Gamma_T$, verifies the asymptotics $T\approx I$ as $z\to \infty$ and satisfies the jump $T_+=T_-J_T$ along $\Gamma_T$, with a jump matrix $J_T$ that can be computed explicitly from $J_T=Y_-^{-1}J_YY_+$.

	The second transformation is referred to in the literature as the opening of the lenses. Define $\Gamma_S := \Gamma_T \cup \gamma_1^\pm \cup \gamma_2^\pm$, with $\gamma_k^\pm$ as in Figure~\ref{fig: opening lens}, making sure that $\Gamma_S$ remains symmetric with respect to $i\R$. Then we open lenses $T \mapsto S$ in the standard way,
	\begin{align}\label{eq: S defn}
	S(z) = \begin{cases}
	T(z)(I\mp e^{2n\phi(z)}E_{21}), & z \text{ in the $\pm$-part of the lenses}, \\
	T(z), & \text{otherwise}.
	\end{cases}
	\end{align}
	Then, $S$ is analytic on $\C\setminus \Gamma_S$, verifies $S(z)=I+\Boh(z^{-1})$ as $z\to \infty$, and satisfies the jump $S_+=S_-J_S$ along $\Gamma_S$, with
	\begin{align}\label{eq: j S defn}
	J_S(z) = \begin{cases}
	e^{i\kappa n}E_{12}-e^{-i\kappa n}E_{21}		 & z \in \gamma_1, \\
	e^{-i\kappa n}E_{12}-e^{i\kappa n}E_{21}		& z \in \gamma_2, \\
	(-1)^{n}I+(-1)^ne^{-2n\phi_+(z)}E_{12}	 & z \in \hat{\gamma},\\
	I+e^{2n\phi(z)}E_{21}& z \in \gamma_1^\pm \cup \gamma_2^\pm.
	\end{cases}
	\end{align}
	
	\begin{figure}[t]
		\centering
		\begin{tikzpicture}[scale=1.5]
		\draw [ultra thick,postaction={mid arrow=black}] (-2,0) to [bend left=12] (-1.75, 1.2);
		\draw [ultra thick,postaction={mid arrow=black}] (-2,0) to [bend left=89] (-1.75, 1.2);
		\draw [ultra thick,postaction={mid arrow=black}] (-2,0) to [bend right=60] (-1.75, 1.2);
		\node [left] at (-2.4, 0.7) {$\gamma_1^+$};
		\node [right] at (-1.5, 0.7) {$\gamma_1^-$};
		
		\draw [ultra thick,postaction={mid arrow=black}] (1.75,1.2) to [bend left=12] (2, 0);
		\draw [ultra thick,postaction={mid arrow=black}] (1.75,1.2) to [bend left=89] (2, 0);
		\draw [ultra thick,postaction={mid arrow=black}] (1.75,1.2) to [bend right=60] (2, 0);
		\node [right] at (2.4, 0.7) {$\gamma_2^+$};
		\node [left] at (1.5, 0.7) {$\gamma_2^-$};
		
		\draw [ultra thick, dashed,postaction={mid arrow=black}] (-1.75,1.2) to [bend left =50] (1.75,1.2);
		
		\draw [fill] (1.75,1.2) circle [radius=0.08];
		\draw [fill] (-1.75,1.2) circle [radius=0.08];
		\draw [fill] (2,0) circle [radius=0.08];
		\draw [fill] (-2,0) circle [radius=0.08];
		\node [below] at (-2,-0.15) {$-1$};
		\node [below] at (2,-0.15) {$1$};
		\node [above] at (0,2.1) {$\hat{\gamma}$};
		\end{tikzpicture}
		\caption{The contour $\Gamma_S$.}
		\label{fig: opening lens}
	\end{figure}
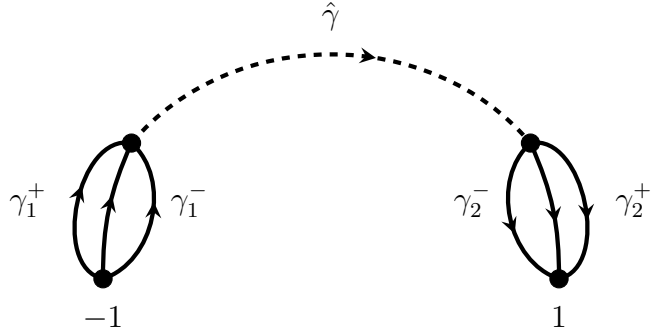
	
	From the inequalities in \eqref{eq: properties phi 1}--\eqref{eq: properties phi 2} we learn that many of these jumps are exponentially decaying. We expect that as $n\to \infty$, the matrix $S$ will converge to the solution of a model Riemann-Hilbert problem which is obtained by dropping these exponentially decaying terms of $J_S$. This idea is made rigorous in the next sections.

\section{Construction of the Global Parametrix}\label{sec: global parametrix}
	When we drop the exponentially decaying terms in \eqref{eq: j S defn}, we arrive at the following model Riemann-Hilbert Problem:
	\begin{subequations}\label{eq:model RHP}
		\begin{alignat}{2}
		&M(z) \text{ is analytic in } \mathbb{C}\backslash\left(\gamma\cup\hat{\gamma}\right),\label{eq: M1a} \\
		&M_+(z) = M_-(z)J_M(z), \qquad\qquad  && z \in \gamma\cup\hat{\gamma}, \label{eq: M jump}\\
		&M(z)= I+\mathcal{O}\left(\frac{1}{z}\right), \qquad && z \to \infty. 
		\end{alignat}
	\end{subequations}
	
	In the Riemann-Hilbert problem above, we also require that $M$ has at worst fourth root singularities at the endpoints of $\gamma_1\cup\gamma_2$, and the jump matrix is explicitly given by
	\begin{equation}\label{def: jumps M}
	J_M(z)=
	\begin{cases}
	e^{i\kappa n}E_{12}-e^{-i\kappa n}E_{21}		 & z \in \gamma_1, \\
	e^{-i\kappa n}E_{12}-e^{i\kappa n}E_{21}		& z \in \gamma_2, \\
	(-1)^{n}I					 & z \in \hat{\gamma}.
	\end{cases}
	\end{equation}
	
	The model RHP above differs from the standard RHP that one usually encounters when solving problems related to the asymptotics of orthogonal polynomials in that the jumps across the different cuts are different and have entries with nonzero imaginary component. 
	
	We will address these issues in the remainder of this section, and for sake of clarity we split this construction into four steps. 
	
	The first step is standard, and consists of building a solution to the particular RHP obtained from $M$ by setting $n=0$. The resulting RHP is of the form one typically encounters when attempting to construct the global parametrix for orthogonal polynomials with positive weights on the real line, and yields a matrix function $N$. 
	
	In the second step, we start from an ansatz on how to modify the entries of $N$ so as to obtain the required $M$. This ansatz turns out to produce a system of scalar Riemann-Hilbert problems.
	
	In the third step, we construct the solutions to these scalar Riemann-Hilbert problems in a somewhat explicit way, with the help of meromorphic differentials on the associated Riemann surface. This can be accomplished for generic values of $n$ and $\lambda$ and is split into two sections. Our ideas are much inspired from \cite{kuijlaars2011global}.
	
	In the fourth and final step we analyze this non-degeneracy condition on $n$ and $\lambda$, showing that the construction is valid as long as either $n$ is even or $n$ is odd and $(n,\lambda)$ is not on the exceptional set $\Theta^*_\varepsilon$ mentioned in the Introduction.
	
	\subsection{Step One: Construction of Simplified Parametrix} 
	
	We first set $n=0$ in \eqref{def: jumps M} and so consider the following RHP:
	\begin{subequations}\label{eq:model RHP N1}
		\begin{alignat}{2}
		&N(z) \text{ is analytic in } \mathbb{C}\backslash\left(\gamma_1\cup \gamma_2\right),\label{eq:model RHP N1a} \\
		&N_+(z) = N_-(z)\begin{pmatrix}
		0 & 1\\
		-1 & 0
		\end{pmatrix}, \qquad && z \in \gamma_1\cup\gamma_2, \label{eq:model RHP N1b} \\
		&N(z)= I+\mathcal{O}\left(\frac{1}{z}\right), \qquad && z \to \infty. \label{eq: N infty}
		\end{alignat}
	\end{subequations}
	The solution to \eqref{eq:model RHP N1} is well known (see for instance \cite{bleher2011lectures}), and is given by
	\begin{equation}\label{eq: N eqn}
	N(z) = \begin{pmatrix}
	\frac{\eta(z)+\eta^{-1}(z)}{2} & \frac{\eta(z)-\eta^{-1}(z)}{-2i}\\
	\frac{\eta(z)-\eta^{-1}(z)}{2i} & \frac{\eta(z)+\eta^{-1}(z)}{2}
	\end{pmatrix},
	\end{equation}
	where 
	\begin{equation}\label{eq: eta eqn}
	\eta(z) = \left(\frac{\left(z+1\right)\left(z-z_*\right)}{\left(z+\overline{z_*}\right)\left(z-1\right)}\right)^{1/4}
	\end{equation}
	with branch cuts on $\gamma_1$ and $\gamma_2$ and the branch of the root chosen so that
	\begin{equation}\label{eq:normalization_eta}
	\lim\limits_{z\to \infty} \eta(z) = 1. 
	\end{equation}
	
	It is important to understand the location of the zeros of $N(z)$, as they will give us an extra degree of freedom when we later modify $N$. 
	
	\begin{prop}\label{prop: N zeros}
		Let
		\begin{equation}\label{def: zero N}
		y_*=\frac{2i}{\lambda(1-x)}
		\end{equation}
		Then
		\begin{equation*}
		\eta\left(y_*\right) - \eta^{-1}\left(y_*\right) = 0,
		\end{equation*}
		and this is the only finite zero of $\eta(z) \pm \eta^{-1}(z)$. In particular, the equation
		\begin{equation*}
		\eta(z) + \eta^{-1}(z) = 0,
		\end{equation*}
		has no finite solutions. 
	\end{prop}
	\begin{proof}
		First note that the solutions to $\eta(z) \pm \eta^{-1}(z)$ are the solutions to 
		\begin{equation*}
		\eta^4(z) = \frac{\left(z+1\right)\left(z-z_*\right)}{\left(z-1\right)\left(z+\overline{z_*}\right)} =1.
		\end{equation*}
		This equation has exactly one finite solution given by 
		\begin{equation*}
		z=y_*:=\frac{2i}{\lambda(1-x)}.
		\end{equation*}
		In other words, either $\eta^2(y_*) = 1$ or $\eta^2(y_*)=-1$, and we will now show that only the former takes place. From \eqref{eq: order xstar}, we have that
		\begin{equation*}
			y_* \to \infty, \qquad \lambda \to \infty. 
		\end{equation*}
		As we normalized the branch of $\eta$ at infinity as in \eqref{eq:normalization_eta}, we have that for $z$ large enough
		\begin{equation*}
			\eta^2(z) = 1 + \frac{1-x}{z} + \mathcal{O}\left(\frac{1}{z^2}\right). 
		\end{equation*}
		In words, the function $\eta^2$ is a continuous function of both $z$ and $\lambda$, and we know that if $\lambda$ is large enough, $y_*$ will be large enough, and in this situation we have that $\eta^2(y_*)$ is close to $1$. In particular, it cannot be close to $-1$, and as we know that $\eta^2(y_*)$ is either $\pm 1$, we can conclude that $\eta^2(y_*) = 1$ for $\lambda$ large enough. However, again using continuity in $\lambda$, we have that $\eta^2(y_*) = 1$ for all $\lambda > \lambda_c$, and not just for $\lambda$ large enough. Then, 
		\begin{equation*}
			\eta^2(y_*) = 1 \iff \eta(y_*) - \eta^{-1}(y_*) = 0,
		\end{equation*}
		concluding the proof. 
	\end{proof}
	
	\subsection{Step Two: Ansatz for the Global Parametrix and Related Scalar RHP's}
	
	We search $M$ in the form
	\begin{equation}\label{eq: M defn}
	M(z) = \begin{pmatrix}
	c_1^{-1} & 0 \\
	0 & c_2^{-1}
	\end{pmatrix}\begin{pmatrix}
	N_{11}(z)v_1^{(1)}(z) & N_{12}(z)v_2^{(1)}(z)\\
	N_{21}(z)v_1^{(2)}(z) & N_{22}(z) v_2^{(2)}(z)
	\end{pmatrix},
	\end{equation} 
	where $c_1$ and $c_2$ are nonzero constants, the functions $v_j^{(k)}$ are yet to be determined, and the notation $N_{ij}$ indicates the $(i,j)$-entry of the matrix $N$. 
	
	Comparing the RHPs for $M$ and $N$, we arrive at the following desired properties for $v_j^{(k)}$:
	
	\begin{enumerate}\label{v prop}
		\item If $j=k$, $v_j^{(k)}$ is analytic on $\mathbb{C}\backslash\left(\gamma\cup\widehat{\gamma}\right)$. 
		\item If $j\not=k$, $v_j^{(k)}$ is analytic on $\mathbb{C}\backslash\left(\gamma\cup\widehat{\gamma}\cup\{y_*\}\right)$, where the singularity at $y_*$ is a simple pole. 
		\item The $v_j^{(k)}$ have the following jumps over $\gamma$, 
		\begin{subequations}\label{eq: v jumps}
			\begin{alignat}{3} 
			&v_{1,\pm}^{(k)}(z) = v_{2,\mp}^{(k)}(z) e^{-in \kappa}, \qquad && z\in \gamma_1,\qquad && k=1,2, \\
			&v_{1,\pm}^{(k)}(z) = v_{2,\mp}^{(k)}(z) e^{in\kappa}, \qquad && z\in \gamma_2,\qquad && k = 1,2, \\
			&v^{(k)}_{j,+}(z) = (-1)^nv^{(k)}_{j,-}(z), \qquad &&z\in \hat{\gamma}, \qquad && k,j = 1,2.
			\end{alignat}
		\end{subequations}
		\item When $j=k$,
		\begin{equation}\label{eq: v infty}
		v_j^{(j)}(z) = c_j + \mathcal{O}\left(\frac{1}{z}\right), \quad z\to \infty, 
		\end{equation}
		for some nonzero constant $c_j$.
		\item If $j\not=k$, then 
		\begin{equation}	\label{eq: v infty nondiag}
		v^{(k)}_j(z)=\Boh(1),\quad z\to \infty. 
		\end{equation}
		\item The functions $v_j^{(k)}$ remain bounded at the endpoints of $\gamma_1\cup\gamma_2$.
	\end{enumerate}
	
	A more rigorous statement of this approach is established in the next result.
	
	\begin{prop}\label{prop: model rhp to scalar}
		If the functions $v_j^{(k)}$ satisfy the conditions (1)--(6) above, then the solution to the RHP \eqref{eq:model RHP} is given by the formula \eqref{eq: M defn}
	\end{prop}
	\begin{proof}
		The condition (1) for $v_j^{(j)}$, together with \eqref{eq:model RHP N1a} imply immediately that $M_{11}$ and $M_{22}$ are analytic on $\C\setminus \left(\gamma\cup\widehat{ \gamma}\right)$. Having in mind condition (2) and Proposition~\ref{prop: N zeros}, we see that for $j\neq k$ the simple pole of $v_{j}^{(k)}$ at $y_*$ cancels with the zero $y_*$ of $N_{kj}$, so the off-diagonal entries of $M$ are also analytic on $\C\setminus \left(\gamma\cup\widehat{ \gamma}\right)$. Condition \ref{eq: M1a} is thus verified.
		
		To verify the appropriate jump conditions, note that for $z \in \gamma_1$, we use the jumps for $v_j^{(k)}$ to get
		\begin{align*}
		M_+(z) &= \begin{pmatrix}
		c_1^{-1} & 0 \\
		0 & c_2^{-1}
		\end{pmatrix}\begin{pmatrix}
		-N_{12,-}(z)v_{2,-}^{(1)}(z) e^{-in\kappa}&  N_{11,-}(z)v_{1,-}^{(1)}(z)e^{in\kappa}\\
		-N_{22,-}(z) v_{2,-}^{(2)}(z) e^{-in\kappa}&	N_{21,-}(z)v_{1,-}^{(2)}(z)e^{in\kappa}
		\end{pmatrix}\\
		&= M_-(z)\begin{pmatrix}
		0 & e^{in\kappa} \\
		-e^{-in\kappa}& 0
		\end{pmatrix}.
		\end{align*}
		
		This is the same as \eqref{eq: M jump} along $\gamma_1$. The jumps over $\gamma_2$ and $\hat{\gamma}$ can be verified in the same manner. The normalization of $M$ at $\infty$ follows from \eqref{eq: v infty} and \eqref{eq: N infty}. Finally, because the $v_j^{(k)}$'s are bounded near the endpoints of $\gamma_1\cup\gamma_2$, the behavior of $M$ near the endpoints of $\gamma_1\cup\gamma_2$ is governed by that of $N$, and as such, $M$ has at worst fourth-root singularities.
	\end{proof}
	
	Therefore, we are left to construct the functions $v_j^{(k)}$, which is accomplished in the remainder of this section.

	\subsection{Construction of the Meromorphic Differentials} 
	
	Let $\mathcal{Q}$ be the genus $1$ Riemann surface associated to the algebraic equation \eqref{eq: spectral curve}. More concretely, this surface $\mathcal Q$ is obtained as the closure of the surface that we get when gluing the two copies
	$$
	\mathcal Q_j=\overline\C\setminus (\gamma_1\cup\gamma_2),\quad j=1,2,
	$$
	along $\gamma_1\cup\gamma_2$ in the usual crosswise manner. On $\mathcal Q$, the global meromorphic solution $\xi$ to the equation \eqref{eq: spectral curve} is well defined by $\restr{\xi}{\mathcal Q_j}=\xi_j$, where $\xi_j$ is the analytic branch of $Q^{1/2}$ on $\mathcal Q_j$ uniquely determined by the asymptotics
	\begin{equation*}
	\xi_j(z)= (-1)^{j}\frac{\lambda i}{2}+\Boh(z^{-1}), \quad z \to \infty, \quad j=1,2.
	\end{equation*}
	In other words, $\xi_1=Q^{1/2}$ is the branch as in \eqref{eq: branch square root} and $\xi_2=-\xi_1$. For notational convenience, we denote by $a^{(j)}$ the copy of the point $a\in \C$ which is on the sheet $\mathcal Q_j$. This is well defined as long as $a$ does not belong to any of the branch cuts $\gamma_1$ and $\gamma_2$, which is enough for our purposes.  
	
	In addition, the canonical homology basis $\{A,B\}$ on $\mathcal{Q}$ is chosen as depicted in Figure~\ref{fig: homology basis}. 
	
	\begin{figure}[t]
		\centering
		\begin{tikzpicture}[scale=1.5]
		\draw [thick] (-2,0) to [bend left=12] (-1.75, 1.2);
		\draw [thick] (1.75,1.2) to [bend left=12] (2, 0);
		\node [right] at (1.9, 0.7) {$\gamma_2$};
		\node [left] at (-1.9, 0.7) {$\gamma_1$};
		
		\draw [ultra thick,postaction={mid arrow=black}] (-2.1,-0.5) to [bend right=90] (-1.75, 1.6);
		\draw [ultra thick,postaction={mid arrow=black}] (-1.75, 1.6) to [bend right=90] (-2.1,-0.5);
		\node [left] at (-2.6, 0.1) {$B$};
		
		\draw [ultra thick,postaction={mid arrow=black}] (-1.9,0.85) to [bend left=120] (1.9, 0.85);
		\draw [dashed, ultra thick,postaction={mid arrow=black}]  (1.9, 0.85) to [bend left=20] (-1.9,0.85);
		\node [right] at (0, 2.05) {$A$};
		
		\draw [fill] (1.75,1.2) circle [radius=0.06];
		\draw [fill] (-1.75,1.2) circle [radius=0.06];
		\draw [fill] (2,0) circle [radius=0.06];
		\draw [fill] (-2,0) circle [radius=0.06];
		
		\end{tikzpicture}
		\caption{The homology basis on $\mathcal{Q }$. The bold contours are on the top sheet of $\mathcal{Q}$, and the dashed contours are on the second sheet of $\mathcal{Q}$. In particular, $B$ is a cycle on the first sheet of $\mathcal{Q}$ that encircles $\gamma_1$ once in the counter-clockwise direction without crossing the imaginary axis. The cycle $A$ starts on the first sheet of $\mathcal{Q}$ and passes through $\gamma_1$ and $\gamma_2$. We also impose that the cycle $A$ is symmetric with respect to the imaginary axis. }
		\label{fig: homology basis}
	\end{figure}
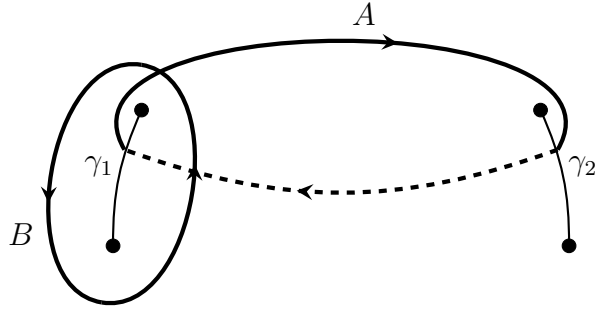

	To construct the functions $v^{(j)}_k$ from the previous section, we will use meromorphic differentials on the Riemann surface $\mathcal Q$. As mentioned in the Introduction, this construction is inspired by a similar construction by Kuijlaars and Mo \cite{kuijlaars2011global}, but here we exploit the symmetry with respect to the imaginary axis in a very explicit way, which also leads to more explicit formulas when compared to \cite{kuijlaars2011global}.
	
	The differential
	$$
	\Lambda_0 = \frac{1}{\xi(z)(z^2-1)}dz
	$$
	linearly generates the space of holomorphic differentials on $\mathcal Q$. For later convenience, for a positive integer $j$ we denote
	\begin{equation}\label{eq: moments differential}
	m^{(j)}_{A}=\frac{1}{2}\oint_A z^j \Lambda_0 = \int_{-\overline{z}_*}^{z_*}\frac{s^j}{\xi(s)(s^2-1)}ds,\quad
	m^{(j)}_{B}=-\frac{1}{2}\oint_B z^j \Lambda_0 = \int_{\gamma_1}\frac{s^j}{\xi_+(s)(s^2-1)}ds.
	\end{equation}
	Because $\Lambda_0$ is, up to a multiplicative constant, the only holomorphic differential on $\mathcal Q$, the Riemann bilinear relations imply that $m_A^{(0)}$ and $m_B^{(0)}$ are nonzero. Also, symmetry (see for instance the Appendix~\ref{appendix: symmetries}) gives 
	\begin{equation}\label{eq: periods holomorphic}
	\re m_A^{(2j)}=\im m_B^{(2j)}=0, \quad \im m_A^{(2j+1)}=\re m_B^{(2j+1)}=0, \qquad j\geq 0.
	\end{equation}
	
	For $a\in i\R$ and $\nu=1,2$, define a meromorphic differential $\Lambda_a^{(\nu)}$ by the formula
	$$
	\Lambda_a^{(\nu)}=\frac{1}{2}\frac{1}{z-a}\left(1+\frac{\xi(a^{(\nu)})(a^2-1)}{\xi(z) (z^2-1)}\right)dz
	$$
	Then the only poles of $\Lambda_a^{(\nu)}$ are $a^{(\nu)}$, $\infty^{(1)}$ and $\infty^{(2)}$, which are all simple, and
	$$
	\res_{\infty^{(1,2)}}\Lambda_a^{(\nu)}=-\frac{1}{2},\quad \res_{a^{(\nu)}}\Lambda_a^{(\nu)}=1.
	$$
	Also, a calculation using symmetry and residues shows that
	\begin{equation}\label{eq: symmetries lambda a}
	\re \oint_{A}\Lambda_a^{(\nu)}=0,\quad \im \oint_B \Lambda_a^{(\nu)}=((-1)^\nu-1)\frac{\pi}{2} .
	\end{equation}
	
	Finally, for $a\in i\R$ and $b\in \R$, $\nu,\varsigma\in \{1,2\}$, set
	\begin{equation}\label{def: meromorphic differential Omega}
	\Omega(a,b)=\Omega(a,b;\nu,\varsigma)=\Lambda_{a}^{(\nu)}-\Lambda^{(\varsigma)}_{y_*}+b\Lambda_0,
	\end{equation}
	where $y_*$ is as in Proposition \ref{prop: N zeros}. 
	
	\begin{prop}\label{prop: symmetry residues periods}
		If $a=y_*$ and $\nu=\varsigma$, then $\Omega(a,b)$ is holomorphic.
		Otherwise, the differential $\Omega(a,b)$ has simple poles at $a^{(\nu)}$ and $y_*^{(\varsigma)}$ with residues $+1$ and $-1$, respectively, and no other poles. Furthermore, in either case
		$$
		\re \oint_A \Omega(a,b)=0,\quad \im \oint_B \Omega(a,b)=((-1)^\nu-(-1)^\varsigma)\frac{\pi }{2}.
		$$
	\end{prop}
	\begin{proof}
		Follow immediately from the properties of $\Lambda_0$ and $\Lambda_a^{(\nu)}$.
	\end{proof}
	
	Moving forward, define $\Psi_B:\R^2\to \R$ by
	$$
	\Psi_B(\tau,b)=\re \oint_B \Omega(i\tau,b)=\re \oint_{B}\left(\Lambda^{(\nu)}_{i\tau}-\Lambda^{(\varsigma)}_{y_*}\right) +b \oint_B \Lambda_0=\re \oint_{B}\left(\Lambda^{(\nu)}_{i\tau}-\Lambda^{(\varsigma)}_{y_*}\right) -2bm_B^{(0)}.
	$$
	Clearly $\Psi_B$ is a linear function of $b$. Furthermore, as the $B$-cycle can be chosen to not intersect $i\R$, $\Psi_B$ is actually a real analytic function of $\tau\in \R$ as well. Thus, the level set determined by $\Psi_B(\tau,b)=0$ is actually the graph of a real analytic function $\tau \mapsto \Psi_B(\tau,b(\tau))$, with
	$$
	b(\tau)=\frac{1}{2m_B^{(0)}}\re \oint_{B}\left(\Lambda^{(\nu)}_{i\tau}-\Lambda^{(\varsigma)}_{y_*}\right).
	$$
	
	Having in mind Proposition~\ref{prop: symmetry residues periods}, we next consider the function $\Psi_A:\R\to \R\backslash \Z$ given by
	\begin{equation}\label{def: psi A}
	\Psi_A(\tau)=\frac{1}{2\pi i}\oint_A \Omega(i\tau,b(\tau)).
	\end{equation}
	For the integration above, we consider the cycle $A$ to be given by straight line segments on $\mathcal Q_1$ and $\mathcal Q_2$ with endpoints $z_*$ and $-\overline z_*$. This way, in principle, $\Psi_A$ is well defined for $\tau \neq \im z_*$ and also for $\tau \neq \im y_*$ in case $\nu=\varsigma$. However, because the residue of $\Omega(i\tau,b)$ at $a^{(\nu)}$ is $1$ we actually get that $\Psi_A$ remains well defined and continuous, as a function with values on $\R\setminus \Z$, when $\tau\to \im z_*$. Also, when $\nu=\varsigma$, we immediately see that $b(\tau)\to 0$ as $\tau\to\im y_*$, and a simple calculation shows that $\Psi_A$ remains continuous with $\Psi_A(\im y_*)=0$, and $\Omega(i\tau,b(\tau))$ reduces to the null meromorphic differential.
	
	In conclusion, $\Psi_A$ is a continuous function from $\R$ to $\R\setminus \Z$.
	
	\begin{prop}\label{prop: injective}
		The function $\Psi_A$ is injective.
	\end{prop}
	\begin{proof}
		Take two values $\tau_1$ and $\tau_2$ for which $\Psi_A(\tau_1)=\Psi_A(\tau_2)$. This means that the $A$-period of the difference $\widehat \Omega = \Omega(i\tau_1,b(\tau_1))-\Omega(i\tau_2,b(\tau_2))$ vanishes. Also, having in mind the definition of $b(\tau)$ and Proposition~\ref{prop: symmetry residues periods}, we get that its $B$-period vanishes as well.
		
		To get to a contradiction, assume that $\tau_1\neq \tau_2$. Then $\widehat \Omega$ has residues $+1$ and $-1$ at $i\tau_1^{(\nu)}$ and $i\tau_2^{(\nu)}$. Hence, for a fixed base point $P_0\in \mathcal Q$, the function 
		$$
		P\mapsto \exp\left(\int_{P_0}^{P}\widehat\Omega \right)
		$$ 
		is a well-defined meromorphic function on $\mathcal Q$ with only one pole, namely at $i\tau_2^{(\nu)}$, and nonzero residue. But since the residues of any meromorphic function have to add to $0$, this is not possible.
	\end{proof}
	
	We also need to compute the limit of $\Psi_A$ as $\tau\to\infty$. To do so, first notice that $\Lambda_{y_*}^{(\varsigma)}$ and $\Lambda_0$ do not depend on $\tau$. To analyze $\Lambda_{i\tau}^{(\nu)}$, we use \eqref{eq: branch square root} to derive the asymptotics
	$$
	\frac{1}{2}\frac{1}{z-a}\frac{\xi(a^{(\nu)})}{\xi(z)}\frac{a^2-1}{z^2-1}=(-1)^{\nu-1}\left(\frac{i a \lambda}{4}+\frac{i\lambda z}{4}+\frac{1}{2}\right)\frac{1}{\xi(z)(z^2-1)}+\Boh(a^{-1}),\quad a\to \infty,
	$$
	which is valid with uniform error term for $z$ in compacts. With $a=i\tau$ this, in turn, implies that
	\begin{multline*}
	\oint_{A,B} \Lambda_{i\tau}^{(\nu)}= \frac{1}{2}\oint_{A,B}\frac{1}{z-a}\frac{\xi(a^{(\nu)})}{\xi(z)}\frac{a^2-1}{z^2-1} dz \\ =\pm (-1)^{\nu-1}\left( \frac{\tau \lambda}{2} m_{A,B}^{(0)}+m_{A,B}^{(0)}+\frac{i \lambda}{2} m_{A,B}^{(1)}\right)+\Boh(\tau^{-1}),\quad \tau\to\infty,
	\end{multline*}
	with the $+$-sign for $A$ and $-$-sign for $B$. Thus,
	$$
	b(\tau)=-\frac{1}{2m_B^{(0)}}\re\oint_B \Lambda_{y_*}^{(\varsigma)} +\frac{(-1)^{\nu}}{4}\left(\tau \lambda+2\right)+\frac{(-1)^{\nu}}{4} \frac{i\lambda m_B^{(1)}}{m_B^{(0)}}+\Boh(\tau^{-1}),\quad \tau \to \infty, 
	$$
	and then
	$$
	\oint_A \Omega(i\tau,b(\tau)) = -\oint_A \Lambda_{y_*}^{(\varsigma)}-\frac{m_A^{(0)}}{m_B^{(0)}}\re \oint_B \Lambda_{y_*}^{(\varsigma)} +(-1)^\nu \frac{i\lambda}{2}\left(\frac{m_A^{(0)}m_B^{(1)}}{m_B^{(0)}}-m_A^{(1)}\right)+\Boh(\tau^{-1}).
	$$
	From the symmetries \eqref{eq: periods holomorphic}--\eqref{eq: symmetries lambda a}, the right-hand side of the above formula is purely imaginary, so defining $c=c(\lambda;\nu,\varsigma)$ through
	\begin{equation*}
	\lim_{\tau\to \infty}\oint_A \Omega_{i\tau}^{(\nu)}=2\pi i\lim_{\tau\to \infty} \Psi_A(\tau)=i c(\lambda;\nu,\varsigma),
	\end{equation*}
	this constant $c$ is purely real and explicitly given by
	\begin{equation}\label{def: theta divisor}
	c=c(\lambda;\nu,\varsigma):=\im\left(-\oint_A \Lambda_{y_*}^{(\varsigma)}-\frac{m_A^{(0)}}{m_B^{(0)}}\re \oint_B \Lambda_{y_*}^{(\varsigma)} +(-1)^\nu \frac{i\lambda}{2}\left(\frac{m_A^{(0)}m_B^{(1)}}{m_B^{(0)}}-m_A^{(1)}\right)\right).
	\end{equation}
	We emphasize that $c$ depends on $y_*$, $\varsigma$, $\nu$ and obviously $\lambda$, but not on $a$.
	
	As a consequence of Proposition \ref{prop: injective}, we see that $\Psi_A$ is a bijection between $\R\cup \{\infty\}$ and an interval of the form $[q,q+1]$. As such, this means that for any $\alpha\in \R\setminus \{ c\}$ we can always choose $\tau$ so that the $A$-period of $\Omega(i\tau,b(\tau))$ equals $2\pi i\alpha$. We summarize the findings of this section as the next result.
	
	\begin{prop}\label{prop: periods injective}
		For $\mu,\varsigma\in \{1,2\}$, $\lambda>\lambda_c$, and $a\in i\R$, $b\in \R$, with $a\neq y_*$ in case $\nu=\varsigma$, the meromorphic differential $\Omega(a,b)$ in \eqref{def: meromorphic differential Omega} has simple poles at $a^{(\nu)}$ and $y_*^{(\varsigma)}$ with residues $+1$ and $-1$, respectively, and is elsewhere holomorphic.
		
		Furthermore, if $\alpha \in \R$ is such that $\alpha-c\in \R\setminus 2\pi\Z$, then there exist unique $a\in i\R$ and $b\in \R$ for which
		$$
		\oint_{A} \Omega(a,b)= \alpha i ,\quad \oint_B \Omega(a,b)=((-1)^\nu-(-1)^\varsigma)\frac{\pi i}{2},
		$$
		where these identities are understood modulo $2\pi i \Z$.
	\end{prop}

	\subsection{Step Three: Solving the Scalar RHP's}
	When the conditions
	\begin{equation}\label{eq: degeneration n kappa}
	n \in 2\Z ,\quad \kappa\in \Z,
	\end{equation}
	take place, the matrix $N$ becomes a solution to the model RHP \eqref{eq:model RHP}, so during this section we can assume the degeneration \eqref{eq: degeneration n kappa} does not occur.
	
	We will now use the meromorphic differential $\Omega(a,b)=\Omega(a,b;\nu,\varsigma_n(\nu))$ as described in the previous section for a special choice of periods, whose existence will be guaranteed by Proposition~\ref{prop: periods injective}. To do so, for $k=1,2$ and $n\in \Z_+$ define numbers $\nu,\varsigma\in \{1,2\}$ by
	\begin{equation}\label{def: varsigma}
	\nu=\nu_{n,k}=n+k+1 \mod 2,\quad \mbox{and} \quad
	\varsigma=\varsigma_k= k+1 \mod 2.
	\end{equation}
	The definition of $\nu$ and $\varsigma$ as above is chosen so as to place the poles of 
	\begin{equation}\label{def: meromorphic differential n k}
	\Omega_n^{(k)}:=\Omega(a,b;\nu_{n,k},\varsigma_k)
	\end{equation}
	in very specific positions. When $n$ is even, the poles $a^{(\nu)}$ and $y_*^{(\nu)}$ are on the same sheet $\mathcal Q_{k+1}$, whereas when $n$ is odd the pole $a^{(\nu)}$ is on the sheet $\mathcal Q_k$ and the pole $y_*^{(\varsigma)}$ is on the other sheet $\mathcal Q_{k+1}$. Obviously, in these comments we identified $\mathcal Q_3=\mathcal Q_1$.
	
	We now choose $a=i\tau$ and $b$ in \eqref{def: meromorphic differential n k} in a very specific way which is formally introduced as the next result.
	
	\begin{corollary}
		Fix $k\in \{1,2\}$ and $n\in \Z_+$. Let $\nu=\nu_{n,k}$ and $\varsigma=\varsigma_{k}$ be as above and $c=c(n,k)$ and $\kappa$ be as in \eqref{def: constant kappa} and \eqref{def: theta divisor}, respectively. Suppose that $2n\kappa-c\in \R\setminus 2\pi\Z$. Then there exists a meromorphic differential $\Omega_n^{(k)}$ with the following properties.
		\begin{enumerate}[(i)]
			\item The only poles of $\Omega_n^{(k)}$, which are simple, are at $y_*^{(\varsigma)}$ as in \eqref{def: zero N} and at another point $a_*^{(\nu)}\in i\R$, with
			$$
			\res_{a_*^{(\nu)}}\Omega_n^{(k)}=1,\qquad \res_{y^{(\varsigma)}}\Omega_n^{(k)}=-1.
			$$
			\item The periods of $\Omega_n^{(k)}$ are
			\begin{equation}\label{eq: periods meromorphic}
			\oint_A \Omega_n^{(k)}= 2n\kappa i,\quad \oint_B \Omega_n^{(k)}=n\pi i
			\end{equation}
			where these identities are understood modulo $2\pi i \Z$.

		\end{enumerate}
	\end{corollary}
	\begin{proof}
		It is a consequence of Propositions~\ref{prop: symmetry residues periods} and \ref{prop: periods injective}. We should also remark that because we are assuming \eqref{eq: degeneration n kappa} does not occur, the poles $y_*^{(\varsigma)}$ and $a_*^{(\nu)}$ never coincide, that is, the degeneration $a_*=y_*$ and $\nu=\varsigma$ in Propositions~\ref{prop: symmetry residues periods} and \ref{prop: periods injective} never takes place.
	\end{proof}
	
	Set $\Gamma := \left(-\infty, -1\right)\cup\gamma\cup\widehat{ \gamma}\cup\left(1, \infty\right)$. For $z\in \C\setminus (\Gamma\cup \{a_*,y_*\})$ define
	\begin{equation}\label{def: u1}
	u^{(k)}_1(z) = \int_1^z \Omega_n^{(k)},\quad k=1,2.
	\end{equation}
	The path of integration always stays on the first sheet $\mathcal Q_1$ and is defined as follows. For $z$ lying above $\Gamma$, the path of integration connects $1$ to $z$ without crossing $\Gamma$, except of course at the initial point $1$. For $z$ lying below $\Gamma$, the path of integration emanates upwards from $1$ and moves to the region below $\Gamma$ crossing the interval $\left(-\infty, -1\right)$. The path then remains below $\Gamma$ until reaching $z$. 
	
	Next, for $z\in \mathcal \C\setminus(\Gamma\cup \{a_*,y_*\})$ we define
	\begin{equation}\label{def: u2}
	u_2^{(k)}(z) = \int_1^z \Omega_n^{(k)} - in\kappa,\quad k=1,2.
	\end{equation}
	The path of integration for $u_2^{(k)}$ is entirely in $\mathcal Q_2$ and specified as follows. For $z$ lying below $\Gamma$, the path of integration connects $1$ to $z$ on the second sheet without crossing $\Gamma$, except at $1$. For $z$ lying above $\Gamma$, the path of integration emanates downwards from $1$ and moves to the region above $\Gamma$ across the interval $\left(-\infty, -1\right)$. The path then remains above $\Gamma$ until meeting $z$. 
	
	Intuitively, one can think of the path of integration for $u_2^{(k)}$ as the mirror image on the other sheet of the path used for $u_1^{(k)}$. Also, because the residues of $\Omega_n^{(k)}$ are $\pm 1$, these functions are well-defined analytic functions modulo $2\pi i$. 
	
	The main properties of $u_j^{(k)}$ are collected in the next result, where all equalities are understood modulo $2\pi i$.
	
	\begin{prop}\label{prop: u jumps}
		\begin{enumerate}[(a)]
			\item 	The functions $u_j^{(k)}$, $j,k = 1,2$, verify the following jumps for $z \in \Gamma$.
			\begin{enumerate}[(i)]
				\item For $x \in \left(-\infty, -1\right)\cup\left(1, \infty\right)$, 
				\begin{equation}\label{eq: u on real line}
				u_{j,+}^{(k)}(x)= u_{j,-}^{(k)}(x).
				\end{equation}
				\item For $z\in \gamma_j$, $j=1,2$, 
				\begin{equation}
				u_{1,\pm}^{(k)}(z) = u_{2,\mp}^{(k)}(z)+(-1)^{j}in\kappa. \label{eq: u on g1}
				\end{equation}
				\item For $z\in\hat{\gamma}$, 
				\begin{equation*}
				u_{j,+}^{(k)}(z) = u_{j,-}^{(k)}(z)+i\pi n.\label{eq: u on gap}
				\end{equation*}
			\end{enumerate}
			
			\item 	For some constants $k_{j,k} \in \mathbb{C}$ the asymptotic behavior
			\begin{equation}\label{eq: u infinity}
			u_j^{(k)}(z) = k_{j,k} + \mathcal{O}\left(\frac{1}{z}\right), \qquad z \to \infty, 
			\end{equation}	
			holds. 
			
			\item	Furthermore, the function $u_j^{(j)}$ is analytic near $y_*$, whereas when $j\not=k$
			\begin{equation}\label{eq: u2 near ys}
			u_j^{(k)}(z) = -\log\left(z-y_*\right) + \mathcal{O}\left(1\right), \qquad z \to y_*.
			\end{equation}
			
			\item	Finally, the behavior near $a_*$ is as follows. 
			\begin{enumerate}[$\bullet$]
				\item For $n$ even, the function $u_{j}^{(j)}$ is analytic near $a_*$, whereas for $j\neq k$
				\begin{equation}\label{eq: u near a}
				u_j^{(k)}(z) = \log\left(z-a\right) + \mathcal{O}\left(1\right), \qquad z \to a.
				\end{equation}	
				
				\item For $n$ odd, when $j\neq k$ the function $u_{j}^{(k)}$ is analytic near $a_*$, whereas
				\begin{equation}\label{eq: u near a n odd}
				u_j^{(j)}(z) = \log\left(z-a\right) + \mathcal{O}\left(1\right), \qquad z \to a.
				\end{equation}
			\end{enumerate}
		\end{enumerate}
	\end{prop}
	\begin{proof}
		First let $x \in (1, \infty)$. Then
		\begin{equation*}
		u^{(k)}_{1,+}(x)-u^{(k)}_{1,-}(x) = \oint_C \Omega_n^{(k)},
		\end{equation*}
		where $C$ is a clockwise loop on the first sheet of $\mathcal{Q}$ surrounding both $\gamma_1$ and $\gamma_2$. By transferring this loop to infinity, we have that
		\begin{equation*}
		u^{(\nu)}_{1,+}(x) = u^{(\nu)}_{1,-}(x), \qquad x \in \left(1,\infty\right).
		\end{equation*}
		The deformation of $C$ to infinity may pick up residue contributions from the poles $a$ and $y_*$ depending on their locations, but as all residues are $\pm 1$, this contribution will only contribute a factor of $2\pi i$. 
		
		Now let $x\in \left(-\infty, 1\right)$. Then, 
		\begin{equation*}
		u^{(k)}_{1,+}(x) - u^{(k)}_{1,-}(x) = \oint_C \Omega_n^{(k)},\qquad x \in \left(-\infty,1\right).
		\end{equation*}
		where now $C$ is contractible within $\mathcal{Q}$, so that \eqref{eq: u on real line} holds for $j=1$. In a similar fashion we compute \eqref{eq: u on real line} for $j=2$. 
		
		Next, take $z \in \gamma_1$. In this case,
		\begin{equation*}
		u^{(k)}_{1,+}(z)-u^{(k)}_{2,-}(z) = \oint_C \Omega_n^{(k)} +i\kappa n, 
		\end{equation*}
		where $C$ can be deformed to the cycle $-A$, so this integral can be computed using \eqref{eq: periods meromorphic} and yields \ref{eq: u on g1} for $j=1$. The case $j=2$, as well as $z\in \gamma_2$, follow analogously.
		
		For the final jump, take $z \in \hat{\gamma}$. Then 
		\begin{equation*}
		u^{(k)}_{1,+}(z) -u^{(k)}_{1,-}(z)= \oint_{-B} \Omega_{a,b}, 
		\end{equation*}
		and this integral is computed using \eqref{eq: periods meromorphic}.
		
		The asymptotic behavior \eqref{eq: u infinity} follows from the fact that $\Omega^{(k)}$ is regular near infinity on the both sheets of $\mathcal{Q}$.

		Next, as $\Omega^{(k)}$ has a simple pole of residue $-1$ at $y_*^{(j)}$ with $k\neq j$, we know that
		\begin{equation*}
		\Omega^{(k)} = \left(-\frac{1}{z-y_*} + \mathcal{O}\left(1\right)\right)\, dz, \qquad z \to y_*^{(j)}, 
		\end{equation*}
		so upon integration, we have \eqref{eq: u2 near ys}. A similar argument also provides provides \eqref{eq: u near a}--\eqref{eq: u near a n odd}.
	\end{proof}

	Finally, we define
	\begin{equation}\label{eq: vj defn}
	v^{(k)}_j(z) = \exp\left(u_j(z)\right),\quad k,j=1,2.
	\end{equation}
	and prove our main result of this section.
	
	\begin{theorem}\label{thm: global parametrix existence}
		For $c$ and $\kappa$ as in \eqref{def: theta divisor} and \eqref{def: constant kappa}, suppose that $2n\kappa-c\in \R\setminus 2\pi\Z$. Then the model Riemann-Hilbert problem \eqref{eq:model RHP} has a unique solution $M=(M_{jk})$. Its entries satisfy
		\begin{enumerate}[(i)]
			\item For $n$ even, $M_{11}$ and $M_{22}$ are never zero, whereas $M_{12}$ and $M_{21}$ have a unique zero at $a_*\in i\R$. 
			\item For $n$ odd, $M_{12}$ and $M_{21}$ are never zero, whereas $M_{11}$ and $M_{22}$ have a unique zero at $a_*\in i\R$.
		\end{enumerate}
	
	\end{theorem}
	
	\begin{proof}
		Uniqueness of $M$ follows in the standard way for Riemann-Hilbert problems, see for instance \cite{deift1999orthogonal}.
		
		To prove existence, we use $v_j^{(k)}$ as in \eqref{eq: vj defn} and set $M$ as in \eqref{eq: M defn}. By Proposition~\ref{prop: model rhp to scalar} it is enough to verify that $v_j^{(k)}$ are solutions to the RHP (1)--(6) in \eqref{v prop} {\it et seq.}. In turn, these scalar RHP conditions follow immediately from Proposition~\ref{prop: u jumps} (a)--(c), and we skip the details.
		
		Finally, the properties of the zeros of $M_{j,k}$ follow from Proposition~\ref{prop: u jumps}--(d).
	\end{proof}
	
	\subsection{Step Four: Analysis of $2n\kappa-c$}\label{section: analysis theta divisor}
	
	The whole construction of the global parametrix that ended up with Theorem~\ref{thm: global parametrix existence} relies on the assumption that $2n\kappa -c \in \R\setminus 2\pi\Z$. It turns out that we can actually remove this restriction, provided that $n$ is even.
	
	Indeed, recall that the meromorphic differential $\Omega^{(k)}_n$ as in \eqref{def: meromorphic differential n k} has a pole at $a^{(n+k+1)}$. Assuming that $2n\kappa-c\in 2\pi\Z$, when we try to match the periods of $\Omega_n^{(k)}$ with the help of the injectivity of $\Psi_A$ in \eqref{def: psi A}, what happens is that the pole at $a=i\tau $ moves to $\infty^{(n+k+1)}$. Nevertheless, the meromorphic differential $\Omega_n^{(k)}$ still has a limit, as one can see performing an asymptotic analysis very similar to the ones that led to \eqref{def: theta divisor}. In such a case, $\Omega_n^{(k)}$ becomes the unique differential  whose only singularities are simple poles at $\infty^{(n+k+1)}$ and $y_*^{(k+1)}$, with residues $+1$ and $-1$, respectively. 
	
	In this way, we can still define $u^{(k)}_j$ as in \eqref{def: u1}--\eqref{def: u2} with a path of integration on the sheet $j$. What will happen now is that when $n$ is even and
	$j\neq k$, as $z\to \infty$ the path of integration extends to the pole at $\infty^{(k+1)}$, and as such $u_j^{(k+1)}\approx -\log z$ as $z\to \infty$. This means that now $v_j^{(k)}(z)\to 0$ as $z\to \infty$, which is no problem at all as the behavior \eqref{eq: v infty nondiag} is still satisfied.
	
	In contrast, when $n$ is odd the path of integration for $u_k^{(k)}$ is now the one that extends to the pole at $\infty^{(k)}$, and consequently $v_j^{(j)}$ vanishes as $z\to \infty$. In such a case, the condition \eqref{eq: v infty} is longer satisfied for a nonzero constant $c_j$.
	
	We can also verify that the condition $2n\kappa-c \in 2\pi \Z$ occurs only for countably many pairs $(n,\lambda)$, if any. Indeed, first notice that $2n\kappa -c$ depends real-analytically on $\lambda>\lambda_c$. Likewise, $c$ does not depend on $k,n$ but only (real-analytically) on $\lambda>\lambda_c$, and we write $2n\kappa-c=2n\kappa(\lambda)-c(\lambda)$. Thus, for any $\ell\in \Z$ and fixed $n\in \N$, the identity
	$$
	2n\kappa(\lambda)-c(\lambda)=2\pi\ell
	$$
	has at most countably many solutions $\lambda$ in $\R$, provided $2n\kappa-c$ is not the constant function. For fixed $n$, a straightforward (but long, so for the sake of simplicity, omitted) calculation shows that
	$$
	2n\kappa(\lambda)-c(\lambda)=g(\lambda)(\alpha+\boh(1)),
	$$
	where $g(\lambda)=\max\{\lambda,\lambda^\delta \log\lambda\}$, $\delta$ as in \eqref{eq: order xstar}, and a constant $\alpha\neq 0$, so that $2n\kappa-c$ is indeed not constant. Now, we fix $\varepsilon >0$ and define the set $\Theta^*_\varepsilon$ as
	\begin{equation}
		\Theta_\varepsilon^* := \left\{(n,\lambda): \text{dist}\left(2n\kappa(\lambda)-c(\lambda),2\pi \Z \right)< \varepsilon\right\}.
	\end{equation}

	By virtue of this discussion, Theorem~\ref{thm: global parametrix existence} is improved to the following form.
	
	\begin{theorem}\label{thm: global parametrix existence_2}
		Fix $\varepsilon>0$. For $n$ even with $\lambda>\lambda_c$ or $n$ odd with $\lambda>\lambda_c$ and $(n,\lambda)\not\in \Theta^*_\varepsilon$, the model Riemann-Hilbert problem \eqref{eq:model RHP} has a unique solution $M=(M_{jk})$. Its diagonal entries satisfy
		\begin{enumerate}[(i)]
			\item For $n$ even, $M_{11}$ and $M_{22}$ are never zero.
			\item For $n$ odd, $M_{11}$ and $M_{22}$ have a unique zero at $a_*\in i\R$ which is simple.
		\end{enumerate}
		Furthermore, the entries of $M$ remain bounded on compacts as $n\to\infty$ with $n$ even or $n$ odd with $(n,\lambda)\notin \Theta_\varepsilon$.
	\end{theorem}

\section{Asymptotic Analysis - Part II}\label{sec: conclusion steepest descent}

	
	In this section we conclude the nonlinear steepest descent analysis.
	
	\subsection{Construction of the Local Parametrices}
	In our analysis, we need to construct four local parametrices, one for each of the endpoints of $\gamma_1\cup\gamma_2$. At the endpoints $z_*$ and $-\overline z_*$ (soft edges) we construct Airy parametrices, whereas at the endpoints $+1$ and $-1$, we construct Bessel parametrices. The construction is now very standard, and for completeness we only give the final form of the solutions near $z=1,z_*$. To construct the other parametrices, we could either proceed analogously or explore the symmetry with respect to the imaginary axis.
	
	\subsubsection{Hard Edges}
	Let $D_4:=D_\delta\left(1\right)$ be a disc centered at $1$ of small radius $\delta>0$. We seek a local parametrix, $P^{(4)}(z)$, defined on $D_4$, which is the solution to the following Riemann-Hilbert problem:
	\begin{subequations}\label{eq: local RHP hard edge}
		\begin{alignat}{2}
		&P^{(4)}(z) \text{ is analytic in } D_1 \backslash \Gamma_S, \\
		&P^{(4)}_+(z)= P^{(4)}_-(z)J_S(z), \qquad && z \in D_4 \cap \Gamma_S, \\
		&P^{(4)}(z) = \left(I + \mathcal{O}\left(\frac{1}{n}\right)\right)M(z), \qquad && \text{uniformly on } \partial D_4 \text{ as } n \to \infty. \label{eq: matching condition hard edge}
		\end{alignat}
	\end{subequations}
	We will also require that $P^{(4)}(z)$ has a continuous extension to $\overline{D_4} \backslash \Gamma_S$ and remains bounded as $z \to 1$. For $\phi$ as in \eqref{eq: phi 4 eqn}, we seek $P^{(4)}(z)$ in the form
	\begin{equation}\label{eq: P 4 defn in terms of U}
	P^{(4)}(z) = U^{(4)}(z) e^{n\phi(z)\sigma_3},
	\end{equation}
	where $U^{(4)}(z)$ solves a Riemann-Hilbert problem which can be explicitly solved using Bessel functions as in \cite{deano2014large,kuijlaars2004riemann}. We describe this solution next.
	
	Set
	\begin{subequations}
		\begin{equation}
		b_1(\zeta) = H_0^{(1)}\left(2\left(-\zeta\right)^{1/2}\right), \qquad b_2(\zeta) = H_0^{(2)}\left(2\left(-\zeta\right)^{1/2}\right),
		\end{equation}
		\begin{equation}
		b_3(\zeta) = I_0\left(2\zeta^{1/2}\right), \qquad b_4(\zeta) = K_0\left(2\zeta^{1/2}\right),
		\end{equation}
	\end{subequations}
	where $I_0$ is the modified Bessel function of the first kind, $K_0$ is the modified Bessel function of the second kind, and $H_0^{(1)}$ and $H_0^{(2)}$ are Hankel functions of the first and second kind, respectively. 
	The Bessel parametrix is defined by
	\begin{align}\label{eq: Bessel parametrix}
	B(\zeta) = \begin{cases}
	\begin{pmatrix}
	\frac{1}{2} b_2\left(\zeta\right) & -\frac{1}{2}b_1\left(\zeta\right) \\
	-\pi \zeta^{1/2} b_2'\left(\zeta\right) & \pi \zeta^{1/2} b_1'\left(\zeta\right)
	\end{pmatrix}, \qquad & -\pi < \arg \zeta < -\frac{2\pi}{3}, \\
	\begin{pmatrix}
	b_3(\zeta) & \frac{i}{\pi}b_4(\zeta) \\
	2\pi i \zeta^{1/2} b_3'(\zeta) & -2\zeta^{1/2} b_4'(\zeta)
	\end{pmatrix}, \qquad & \left|\arg \zeta\right| < \frac{2\pi}{3}, \\
	\begin{pmatrix}
	\frac{1}{2} b_1\left(\zeta\right) & \frac{1}{2}b_2\left(\zeta\right) \\
	\pi \zeta^{1/2} b_1'\left(\zeta\right) & \pi \zeta^{1/2} b_2'\left(\zeta\right)
	\end{pmatrix}, \qquad & \frac{2\pi}{3} < \arg \zeta < \pi.
	\end{cases}
	\end{align}
	Let $f_{n,B}$ being the conformal map
	\begin{equation}\label{eq: Bessel conformal map}
	f_{n,B}(z) = n^2 f_B(z), \qquad \text{where} \qquad f_B(z) = \frac{1}{4}\left(-\phi(z)+\frac{i\kappa}{2}\right)^2.
	\end{equation}
	Recall from \eqref{eq: properties phi 1} that
	\begin{equation}
		\phi_+(z) + \phi_-(z) = i\kappa, \qquad z \in \gamma_1, 
	\end{equation}
	so that 
	\begin{equation}
		\left(-\phi(z) +\frac{i\kappa}{2}\right)_+ = -\left(-\phi(z)+\frac{i\kappa}{2}\right)_-, \qquad z\in\gamma_1,
	\end{equation}
	so that $\left(-\phi +\frac{i\kappa}{2}\right)^2$ is analytic in $D_4\backslash\{1\}$. However, as $\phi(1)=0$, we see that the singularity is removable, so that $f_B$ is actually analytic in all of $D_4$. We see from the definition of $\phi$ in \eqref{eq: phi 4 eqn}, 
	\begin{equation}
		-\phi(z) + \frac{i\kappa}{2} = -\sqrt{\frac{4+\lambda\left(4i+\lambda\left(x^2-1\right)\right)}{2}}\sqrt{z-1} + \mathcal{O}\left(\left(z-1\right)^{3/2}\right), \qquad z\to 1,
	\end{equation}
	so that
	\begin{equation}
		f_B(z) = \frac{4+\lambda\left(4i+\lambda\left(x^2-1\right)\right)}{8}\left(z-1\right) + \mathcal{O}\left(\left(z-1\right)^2\right), \qquad z\to 1. 
	\end{equation}
	As $f'_{n,B}(1)\not=0$, $f_{n,B}$ is indeed a conformal mapping on a neighborhood of 1. With $f_{n,B}$ defined as above, the matrix $U^{(4)}$ is
	\begin{equation}\label{eq: U4 eqn}
	U^{(4)}(z) = E^{(4)}(z) B(f_{n,B}(z)),
	\end{equation}
	where
	\begin{equation}
	E^{(4)}(z) = M(z) e^{-\frac{in\kappa}{2}\sigma_3}L^{(4)}(z)^{-1}, 
	\quad			
	L^{(4)}(z) := \frac{1}{\sqrt{2}} \left(2\pi n\right)^{-\sigma_3/2} f_B(z)^{-\sigma_3/4}
	\begin{pmatrix}
	1 & i \\
	i & 1
	\end{pmatrix}.
	\end{equation}
	
	The parametrix $P^{(1)}$ in a small neighborhood $D_1$ of the hard edge $-1$ can be constructed by exploring the symmetry w.r.t. the imaginary axis, which leads to
	$$
	P^{(1)}(z):=\overline{P^{(4)}(-\overline z)},\quad z\in D_1.
	$$
	
	\subsubsection{Soft Edges}
	Let $D_2:=D_\delta\left(-\overline{z_*}\right)$ be a small disc centered at $-\overline{z_*}$ of radius $\delta>0$. We seek a local parametrix, $P^{(2)}(z)$, defined on $D_2$, which is the solution to the following Riemann-Hilbert problem:
	\begin{subequations}\label{eq: local RHP soft edge}
		\begin{alignat}{2}
		&P^{(2)}(z) \text{ is analytic in } D_2 \backslash \Gamma_S, \\
		&P^{(2)}_+(z)= P^{(2)}_-(z)J_S(z), \qquad && z \in D_2 \cap \Gamma_S, \\
		&P^{(2)}(z) = \left(I + \mathcal{O}\left(\frac{1}{n}\right)\right)M(z), \qquad && \text{uniformly on } \partial D_2 \text{ as } n \to \infty. \label{eq: matching condition soft edge}
		\end{alignat}
	\end{subequations}
	We also require that $P^{(2)}(z)$ has a continuous extension to $\overline{D_2} \backslash \Gamma_S$ and remains bounded as $z \to -\overline{z_*}$. 
	From standard methods, it is natural to seek a solution of the form		
	\begin{equation}\label{eq: U2 eqn}
	P^{(2)}(z) = U^{(2)}(z) e^{n\left[\phi^{(2)}(z)+\frac{i\pi}{2}\right]\sigma_3},
	\end{equation}		
	where
	\begin{equation}\label{eq: phi 2 defn}
	\phi^{(2)}(z) := \int_{-\overline{z_*}}^{z} Q^{1/2}(s)\, ds - \frac{i \kappa}{2}. 
	\end{equation}
	
	Then the RHP for $P^{(2)}$ reduces to a RHP for $U$, which in turn can be solved explicitly using Airy functions in a now standard way. Set $\omega:=\exp\left(2\pi i/3\right)$,
	\begin{equation}\label{eq: Airy functions}
	y_0(z) := \text{Ai}(z), \qquad y_1(z):=\omega\text{Ai}(\omega z), \qquad y_2(z):=\omega^2 \text{Ai}(\omega^2 z),
	\end{equation}
	where $\text{Ai}$ is the Airy function, and define the Airy parametrix as in \cite{bleher2011lectures} by
	\begin{align}\label{eq: Airy parametrix}
	A(z) = \begin{cases}
	\begin{pmatrix}
	y_0(z) & -y_2(z) \\
	y_0'(z) & -y_2'(z) 
	\end{pmatrix}, \qquad & \arg z \in \left(0, \frac{2\pi}{3}\right), \\
	\begin{pmatrix}
	-y_1(z) & -y_2(z) \\
	-y_1'(z) & -y_2'(z) 
	\end{pmatrix}, \qquad & \arg z \in \left(\frac{2\pi}{3},\pi\right),\\
	\begin{pmatrix}
	-y_2(z) & y_1(z) \\
	-y_2'(z) & y_1'(z) 
	\end{pmatrix}, \qquad & \arg z \in \left(-\pi, -\frac{2\pi}{3}\right),\\
	\begin{pmatrix}
	y_0(z) & y_1(z) \\
	y_0'(z) & y_1'(z) 
	\end{pmatrix}, \qquad & \arg z \in \left(-\frac{2\pi}{3},0\right).
	\end{cases}
	\end{align}
	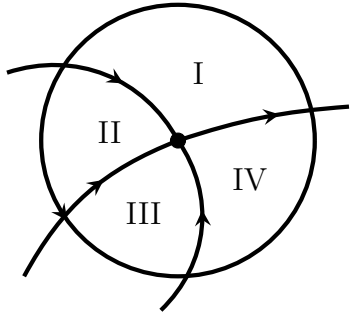
\begin{figure}[t]
		\centering
		\begin{tikzpicture}[scale=4.5]
		\draw[ultra thick, postaction={mid arrow=black}] (0,0) arc (0:360:0.4 and 0.4);
		\draw [fill] (-0.4,0) circle [radius=0.022];
		
		\draw [ultra thick,postaction={mid arrow=black}] (-0.45,-0.5) to [bend right=40] (-0.4, 0);
		\draw [ultra thick,postaction={mid arrow=black}] (-0.9,0.2) to [bend left=40] (-0.4, 0);
		\draw [ultra thick,postaction={mid arrow=black}] (-0.85,-0.4) to [bend left=20] (-0.4, 0);
		\draw [ultra thick,postaction={mid arrow=black}] (-0.4,0) to [bend left=7] (0.1, 0.1);
		
		\node at (-0.34,.20) {I};
		\node at (-0.6,.025) {II};
		\node at (-0.5,-.215) {III};
		\node at (-0.19,-.115) {IV};
		\end{tikzpicture}
		\caption{Definition of Sectors I, II, III, and IV within $D_2$.}
		\label{fig: defn of I II III IIV}
	\end{figure}
	Then $U^{(2)}$ takes the form
	\begin{equation}\label{eq: U2 eqn Airy}
	U^{(2)}(z) = E_n^{(2)}(z)A(f_{n,A}(z)), 
	\end{equation}
	where $A$ is the Airy parametrix as above, $f_{n,A}$ is the conformal map given by
	\begin{equation}\label{eq: Airy conformal map}
	f_{n,A}(z) = n^{2/3}f_A(z), \qquad f_A(z) = \left[\frac{3}{2}\left(\phi^{(2)}(z)+\frac{i\kappa}{2}\right)\right]^{2/3}. 
	\end{equation}
	From \eqref{eq: phi 2 defn}, we can write for $z$ in a neighborhood $-\overline{z_*}$,
	\begin{equation}\label{eq: conf map airy}
		\phi^{(2)}(z) + \frac{i\kappa}{2} = \frac{2}{3}\left(z+\overline{z_*}\right)^{3/2} h(z),
	\end{equation}
	where the cut for $(z+\overline{z_*})^{3/2}$ is taken on $\gamma_1$ and $h(z)$ is analytic in a neighborhood of $-\overline{z_*}$, with the following Taylor expansion near $-\overline{z_*}$,
	\begin{equation}
		h(z) = h(-\overline{z_*}) + h'(-\overline{z_*})\left(z+\overline{z_*}\right) + \dots,
	\end{equation}
	and 
	\begin{equation}
		h(-\overline{z_*}) =\frac{\sqrt{2x\left(4+4ix\lambda + \lambda^2(1-x^2)\right)}}{\lambda} \not= 0.
	\end{equation}
	From \eqref{eq: conf map airy}, we see that $f_A$ has no jumps within $D_2$, and as $f_A(-\overline{z_*}) = 0$, we can conclude the singularity at $-\overline{z_*}$ is again removable. As $f_A'(-\overline{z_*}) = h(-\overline{z_*}) \not= 0$, we see that $f_A$ is indeed a conformal mapping onto a neighborhood of $0$. 
	The prefactor $E_n^{(2)}$ is 
	\begin{align}\label{eq: En eqn}
	E_n^{(2)}(z) =\begin{cases}
	M(z) e^{-n\left[\frac{i\pi}{2}-\frac{i\kappa}{2}\right]\sigma_3}L_N(z)^{-1}, \qquad & z\in \text{I}, \text{II}, \\
	M(z) e^{-n\left[-\frac{i\pi}{2}-\frac{i\kappa}{2}\right]\sigma_3}L_N(z)^{-1}, \qquad & z\in \text{III}, \text{IV},
	\end{cases}
	\end{align}
	where Sectors I, II, III, and IV are defined in Figure~\ref{fig: defn of I II III IIV}, and
	\begin{equation*}
	L_N(z) = \frac{1}{2\sqrt{\pi}}n^{-\sigma_3/6}f_A(z)^{-\sigma_3/4}\begin{pmatrix}
	1 & i \\
	-1 & i
	\end{pmatrix}.
	\end{equation*}
	In the formulas above, all the roots are taken to be the principal branches.	
	
	In a similar fashion to the hard edge scenario, we compute the local parametrix $P^{(3)}$ in a small neighborhood $D_3$ of $z_*$	via symmetry, which leads to
	$$
	P^{(3)}(z):=\overline{P^{(3)}(-\overline z)},\quad z\in D_3.
	$$

	
	\subsection{Final Transformation}
	As the last transformation of the steepest descent method, we define
	\begin{align}\label{eq: R eqn}
	R(z) = \begin{cases}
	S(z)M(z)^{-1}, \qquad &z \in \mathbb{C}\backslash\left(\bigcup\limits_{i=1}^{4} D_i \cup \Gamma_S\right),\\
	S(z)P^{(j)}(z)^{-1}, \qquad &z \in D_j\backslash\Gamma_S, \qquad  j=1,\dots,4.
	\end{cases}
	\end{align}
	Then $R$ solves a Riemann-Hilbert problem of the form,
	\begin{subequations}\label{eq: R RHP}
		\begin{alignat}{2}
		&R(z) \text{ is analytic in } \mathbb{C}\backslash \Gamma_R, \\
		&R_+(z) = R_-(z) J_R(z), \qquad\qquad &&z\in \Gamma_R, \\
		&R(z) = I + \mathcal{O}\left(\frac{1}{z}\right), \qquad &&z \to \infty,
		\end{alignat}
	\end{subequations}
	where the contour $\Gamma_R$ is depicted in Figure~\ref{fig: Gamma R}, and the jump matrix $J_R$ that satisfies
	\begin{align*}
	J_R(z) = \begin{cases}
	I+\mathcal{O}\left(e^{-cn}\right), \qquad & z \in \Gamma_R\backslash\left(\bigcup\limits_{i=1}^{4} \partial D_i\right), \\
	I+\mathcal{O}\left(\frac{1}{n}\right), \qquad & z \in \bigcup\limits_{i=1}^{4} \partial D_i,
	\end{cases}
	\end{align*}
	with uniform error terms. The definition of $M$, and of $R$ itself, only requires that $n$ is even or $2n\kappa(\lambda)-c(\lambda)\notin 2\pi \Z$. The condition that $(n,\lambda)\notin \Theta_\varepsilon$ appears here to ensure that the entries of $M$ remain bounded as $n\to\infty$, see Theorem~\ref{thm: global parametrix existence_2} above, and consequently ensure that $J_R$ indeed decays to the identity as claimed.
	
	\begin{figure}[t]
		\centering
		\begin{tikzpicture}[scale=1.5]
		\draw [ultra thick,postaction={mid arrow=black}] (-2.2,0.35) to [bend left=30] (-1.75, 1.5);
		\draw [ultra thick,postaction={mid arrow=black}] (-1.6,0.15) to [bend right=30] (-1.25, 1.25);
		
		\draw [ultra thick,postaction={mid arrow=black}] (1.75,1.5) to [bend left=30] (2.2, 0.35);
		\draw [ultra thick,postaction={mid arrow=black}] (1.25,1.25) to [bend right=30] (1.6, 0.15);
		
		\draw [ultra thick, postaction={mid arrow=black}] (-1.15,2) to [bend left =40] (1.15,2);
		
		\draw[ultra thick, postaction={mid arrow=black}] (-1.6,0) arc (0:360:0.4 and 0.4);
		\draw[ultra thick, postaction={mid arrow=black}] (1.6,0) arc (180:540:0.4 and 0.4);
		\draw[ultra thick, postaction={mid arrow=black}] (-0.95,1.65) arc (0:360:0.4 and 0.4);
		\draw[ultra thick, postaction={mid arrow=black}] (1.75,1.64) arc (0:360:0.4 and 0.4);
		
		\node at (2,0) {$D_4$};
		\node at (-2,0) {$D_1$};
		\node at (-1.35,1.65) {$D_2$};
		\node at (1.35,1.65) {$D_3$};
		\end{tikzpicture}
		\caption{The contour $\Gamma_R$.}
		\label{fig: Gamma R}
	\end{figure}
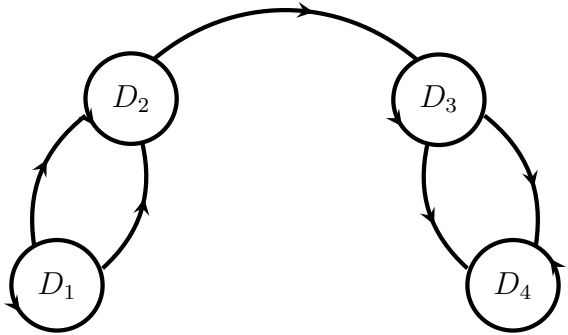
	
	As shown in \cite{bleher2011lectures}, this is enough to assure that
	\begin{equation}\label{eq:identity_convergence_R}
	R(z) = I + \mathcal{O}\left(\frac{1}{n}\right), \qquad n \to \infty,
	\end{equation}
	uniformly for $z \in \mathbb{C}\backslash\Gamma_R$ as $n\to\infty$ with $n$ even or with $(n,\lambda)\notin \Theta_\varepsilon$. We are now left to retrace our steps in the transformations of nonlinear steepest descent to obtain the uniform asymptotic formulas for $p_n^\lambda$.
	
	\section{Asymptotics of the Kissing Polynomials: Discussion on Theorem~\ref{thm: strong asymptotics gamma}}\label{sec: asymptotics}
	The procedure of recovering asymptotics of the orthogonal polynomials from the nonlinear steepest descent analysis of the associated Riemann-Hilbert problem is now standard, so we omit details below, and instead refer the reader to \cite{bleher2011lectures,deano2014large,deift1999orthogonal} for detailed analysis. In particular, tracing back all the transformations $Y\mapsto T\mapsto S\mapsto R$, we arrive at the identity
$$
p_n^{\lambda}(z)=e^{-n\left(\ell -\frac{i\kappa}{2}+\phi(z)-\frac{V(z)}{2}\right)}\left(R_{11}(z)M_{11}(z)+R_{12}(z)M_{21}(z)\right).
$$	
The function $M_{21}$ is analytic away from $\gamma_1\cup\gamma_2$ and according to Theorem~\ref{thm: global parametrix existence} either  one of $M_{11}$ or $M_{21}$ has a zero, depending on the parity of $n$, but not both. We can then use \eqref{eq:identity_convergence_R} to prove the weak convergence claimed by Theorem~\ref{thm: strong asymptotics gamma}, as well as to find the asymptotic formula
	\begin{equation}\label{eq: asymptotic formula 1}
	p_n^\lambda(z)=e^{n\left( \frac{i\kappa}{2}+\frac{V(z)}{2}-l-\phi(z)\right)}\left(M_{11}(z)+\mathcal{O}\left(\frac{1}{n}\right)\right),
	\end{equation}
	the identity above being valid with uniform error term for $z$ on compacts of $\C\setminus(\gamma_1\cup\gamma_2)$. This formula is essentially \eqref{eq: asymptotics outside gamma}, but for the sake of completeness we now rewrite $M_{11}$ in a more explicit form that leads to the formulation \eqref{eq: asymptotics outside gamma} with the functions $\Psi_{n,0}$ and $\Psi_{n,1}$, summarizing our findings in a language that does not require prior knowledge of meromorphic differentials on Riemann surfaces.
	
	In explicit terms, the function $u_1^{(1)}$ in \eqref{def: u1} can be written as
	\begin{multline*}
	u_1^{(1)}(z)= b_*\int_1^z \frac{ds}{Q^{1/2}(s)(s^2-1)} \\
	+ \frac{1}{2}\int_1^z\left[\frac{1}{z-y_*}\left(1-\frac{Q^{1/2}(y_*)(y_*^2-1)}{Q^{1/2}(s)(s^2-1)}\right) + \frac{1}{z-a_*}\left(1+(-1)^{n+1}\frac{Q^{1/2}(a_*)(a_*^2-1)}{Q^{1/2}(s)(s^2-1)}\right) \right]ds,
	\end{multline*}
	where the path of integration emanates from $s=1$ in the upper half-plane and is contained in $\C\setminus\left(\gamma_1\cup\widehat \gamma\cup \gamma_2\cup (1,+\infty)\right)$. Here, $y_*$ is given by \eqref{def: zero N} and the points $a_*$ and $b_*$ are uniquely determined by the requirement that
	$$
	\int_{-\overline z_*}^{z^*}\left(2b_*-\frac{Q^{1/2}(s)}{s-y_*}+(-1)^{n+1}\frac{Q^{1/2}(s)}{s-a_*}\right)\frac{ds}{Q^{1/2}(s)(s^2-1)}=2n\kappa i
	$$
	and
	$$
	\int_{\gamma_1}\left(2b_*-\frac{Q^{1/2}(s)}{s-y_*}+(-1)^{n+1}\frac{Q^{1/2}(s)}{s-a_*}\right)\frac{ds}{Q_+^{1/2}(s)(s^2-1)}=n\pi i,
	$$
	where these identities are valid modulo $2\pi i$. Such $u_1^{(1)}$ is well defined modulo $2\pi i$ on $\C\setminus (\gamma_1\cup\widehat \gamma\cup \gamma_1)$, and we write $u_1^{(1)}=u_0$ for $n$ even and $u_1^{(1)}=u_1$ for $n$ odd. The function $u_0$ has no singularities on its domain of definition, whereas $u_1$ has a unique singularity at $z=a_*$, with behavior as in \eqref{eq: u near a n odd}. Next, for $\eta$ as in \eqref{eq: eta eqn} we set
	$$
	\Psi_{n,j}(z)=M_{11}(z)=\frac{\eta^2(z)+1}{2\eta(z)}e^{u_j(z)},\quad z\in \C\setminus (\gamma_1\cup\gamma_2), \quad j=0,1.
	$$
	Plugging this back into \eqref{eq: asymptotic formula 1} leads to the formulation of Theorem~\ref{thm: strong asymptotics gamma}.

\appendix

\section{Symmetry Relations}\label{appendix: symmetries}

The following two technical lemmas in complex analysis were extensively used throughout the text. Their proofs are straightforward. 

\begin{lemma}
Let $\gamma_1$ be an contour on the left half plane, from $p$ to $q$, and $\gamma_2$ be the contour obtained from $\gamma_1$ upon reflection over the imaginary axis, oriented from $-\overline{q}$ to $-\overline{p}$.

Suppose that a function $f$ satisfies the symmetry relation
$$
\overline{f(s)}=\delta f(-\overline{s}),
$$
where $\delta\in \{+1,-1\}$.
Then
$$
\overline{\int_{\gamma_1} f(s)ds}=\delta\int_{\gamma_2}f(s)ds
$$
In particular, 
$$
\int_{\gamma_1}f(s)ds+\int_{\gamma_2}f(s)ds=
\begin{cases}
2\re \int_{\gamma_1}f(s)ds,& \mbox{if }\; \delta=1, \\
2i \im \int_{\gamma_1}f(s)ds, & \mbox{if } \; \delta=-1.
\end{cases}
$$
\end{lemma}

\begin{lemma}
Let $\widehat{\gamma}$ be a contour symmetric with respect to the imaginary axis and $f$ be as in the previous lemma. Then
$$
\im \int_{\widehat \gamma}f(s)ds=0,\quad \mbox{if }\;\delta=1,\qquad \re \int_{\widehat \gamma}f(s)ds=0,\quad \mbox{if }\; \delta=-1.
$$
\end{lemma}

\bibliographystyle{plain}

\end{document}